\documentclass[12pt]{amsart}
\usepackage{fullpage}
\usepackage{latexsym,amsmath,amsfonts,amsthm,amssymb,color}
\usepackage[all]{xy}
\usepackage{hyperref} %use this when generating PDF
\usepackage{enumitem}
\usepackage[T1]{fontenc}

\usepackage{color}

\usepackage{mathrsfs}
\usepackage[bottom]{footmisc}

\newtheorem{theorem}{Theorem}
\newtheorem{lemma}[theorem]{Lemma}
\newtheorem{proposition}[theorem]{Proposition}
\newtheorem{corollary}[theorem]{Corollary}
\newtheorem{definition}[theorem]{Definition}
\newtheorem{remark}[theorem]{Remark}

\numberwithin{equation}{section}
\numberwithin{theorem}{section}

\newcommand{\R}{{\mathbb R }}
\newcommand{\N}{{\mathbb N }}
\newcommand{\C}{{\mathbb C }}
\newcommand{\Z}{{\mathbb Z }}
\newcommand{\dbar}{\bar \partial}
\newcommand{\qT}{q}
\newcommand{\tT}{\mathbf s}
\newcommand{\Ic}{{\partial}^{\,-1}}
\newcommand{\Icb}{\bar \partial^{-1}}
\newcommand{\Icbk}{\ol{\partial}^{\,-1}_k}
\newcommand{\dnuz}{dz}
\newcommand{\dnuk}{dk}
\newcommand{\dkb}{\frac{\partial}{\partial\ol k}}

\newcommand{\dzb}{\bar \partial}
\newcommand{\dz}{\partial}
\newcommand{\ol}{\overline}
\newcommand{\half}{\frac{1}{2}}
\newcommand{\Hq}{\mathcal H_q}

\begin{document}

\title{A Nonlinear Plancherel Theorem with Applications to Global Well-Posedness
for the Defocusing Davey-Stewartson Equation and to the Inverse
Boundary Value Problem of Calder\'on}

\author{Adrian Nachman}
\author{Idan Regev}
\author{Daniel Tataru}

\address{Department of Mathematics, University of Toronto, Toronto, ON M5S-2E4}
\email{nachman@math.toronto.edu}
\email{idan.regev@utoronto.ca}

\address{Department of Mathematics, University of California, Berkeley, CA 94720}
\email{tataru@math.berkeley.edu}

\begin{abstract}
  We prove a Plancherel theorem for a nonlinear Fourier transform in
  two dimensions arising in the Inverse Scattering method for the
  defocusing Davey-Stewartson II equation.  We then  use it to prove global
  well-posedness and scattering in $L^2$ for defocusing DSII. This Plancherel theorem also
  implies global uniqueness in the inverse boundary value problem of
  Calder\'on in dimension $2$, for conductivities $\sigma>0$ with  $\log \sigma \in \dot H^1$.  The proof
  of the nonlinear Plancherel theorem includes new estimates on
  classical fractional integrals, as well as a new result on
  $L^2$-boundedness of pseudo-differential operators with non-smooth
  symbols, valid in all dimensions.
\end{abstract}

\maketitle \tableofcontents

%%% ======================================================================
\section{Introduction}\label{Section:Intro}
%%% ======================================================================

The Davey-Stewartson equations are a family of nonlinear
Schr\"odinger (NLS) type equations in $2+1$ dimensions, which
model the evolution of weakly nonlinear surface water waves
traveling principally in one direction \cite{DS}. A rigorous
derivation from the water wave problem in the modulational scaling
regime is provided in \cite{CSS}. Depending on two sign choices,
one classifies the Davey-Stewartson systems as elliptic-elliptic,
hyperbolic-elliptic, elliptic-hyperbolic and
hyperbolic-hyperbolic. Within each class there are two nontrivial
choices of parameters to be made.

In this paper we are interested in the Cauchy problem
for the specific  case known as
the defocusing DSII problem. This model belongs to the hyperbolic-elliptic family,
with a special choice for the parameters. The equations have the form
\begin{align}\label{DSIIdefoc}
\begin{cases}
i\partial_t \qT + 2(\dzb^2+\dz^2)\qT+\qT(r+\ol r) = 0\\
\dzb r + \dz(|\qT|^2)=0\\
\qT(0,z) = \qT_0(z).
\end{cases}
\end{align}
Here and throughout the paper we use the notation
\begin{align*}
z = x_1+i x_2
\end{align*}
for points in the plane, and
\begin{align}\label{dzbdz}
\dzb = \half\Big(\frac{\partial}{\partial x_1}+ i
\frac{\partial}{\partial x_2}\Big), \qquad \dz =
\half\Big(\frac{\partial}{\partial x_1}- i
\frac{\partial}{\partial x_2}\Big).
\end{align}

This system (as well as all other DS systems) is mass critical, i.e.
the $L^2_z$ norm of the solution (the mass) is invariant
with respect to the natural scaling associated to it,
\[
\qT(t,z) \to \lambda \qT(\lambda^2 t,\lambda z).
\]
Local well posedness in $L^2$ and global existence for small
initial data have been established for the general family of
Davey-Stewartson equations  in \cite{GuiSau},
\cite{LP93}, \cite{HS95} using dispersive methods. However,
the large data problem has yet to be understood in general.

The defocusing DSII model considered here has the feature that it
is completely integrable, as found in \cite{AH}. In this paper we
use the Inverse Scattering method to investigate the Cauchy
problem in $L^2$ for large initial data.

Precisely, our main goal here will be to prove a Plancherel theorem for a
two-dimensional nonlinear Fourier transform (known as the Scattering
Transform) associated to this system. We then use this result to show
global well-posedness and scattering for (\ref{DSIIdefoc}) for any
initial data in $L^2(\R^2)$, i.e. in the mass-critical
case. Furthermore, the method yields a precise description of the
large-time behaviour of the solutions for any initial data $\qT_0$ in
$L^2(\R^2)$ in terms of its Scattering Transform. The Plancherel
theorem implies completeness of the wave operators (in the sense of
nonlinear scattering theory).

In a different application, we show how this nonlinear Plancherel
theorem also implies global uniqueness for the inverse boundary
value problem of Calder\'on in dimension 2, for conductivities
$\sigma>0$ with  $\log \sigma \in \dot H^1$. We will briefly
recall some of the background for these problems below.

\subsection{The Scattering Transform}
We start with a quick formal definition of the scattering
transform. Given a function $\qT(z)$ on $\R^2\simeq\C$ and
$k\in\C$, solve the two equations
\begin{align}\label{eq_n}
\frac{\partial}{\partial\ol z} m_\pm =  \pm  e_{-k} \qT \ol
{m_\pm}
\end{align}
with $m_\pm(z,k)\rightarrow 1$ as $|z|\rightarrow\infty$. (We use
the notation $e_k(z) = e^{i(zk+\ol{zk})}$ first introduced in
\cite{Nac}). The scattering transform of $\qT$ is then defined as
\begin{align}\label{tT}
\mathcal S\qT(k) = \frac{1}{2\pi i}\int_{\R^2} e_k(z) \ol{\qT(z)}
\Big(m_+(z,k) + m_-(z,k)\Big) dz,
\end{align}
where $dz = dx_1dx_2$ denotes the Lebesgue measure on $\R^2$. Note
that different authors have slightly different conventions.

As seen in \eqref{eq_n}, a key step in the analysis is to be
able to invert d-bar operators $L_q$ of the form
\[
L_\qT u = \dzb u + \qT \ol u
\]
under only the assumption that $q \in L^2$. Using the Sobolev embedding $\dot H^{\half}\subset L^4$,
it is easy to see that $L_\qT$ has the following mapping property:
\[
L_\qT : \dot H^\frac12 \to \dot H^{-\frac12},
\]
where $\dot H^s$ is the homogeneous fractional Sobolev space of order $s\in\R$, defined as
\[
\{f\in\mathscr S\,'\,\Big|\, \hat f \text{ is a measurable function with}\, \|f\|_{\dot H^s}:=\|\,|\cdot|^s\hat f(\cdot)\|_{L^2}<\infty\}.
\] 
Then it is natural to consider the solvability question for
the corresponding inhomogeneous problem
\begin{align}\label{LqEq}
L_q u  = f
\end{align}
in the $\dot H^\frac12 \to \dot H^{-\frac12}$ setting. Our main
result on this problem is as follows:

\begin{theorem} \label{ThmMain}
Let $\qT \in L^2$. Then for each $f \in \dot H^{-\frac12}$ there
exists a unique solution $u \in \dot H^\frac12$ of  the inhomogeneous problem \eqref{LqEq}
with
\begin{equation}\label{Lq-inv}
\| u\|_{\dot H^\frac12} \leq C(\|q\|_{L^2}) \|f\|_{ \dot H^{-\frac12}}.
\end{equation}
\end{theorem}
A key point here is that the  constant depends only on the $L^2$ norm of $q$.
This will later  allow us to show that the solutions of (\ref{eq_n}) are bounded by
constants that depend only on the $L^2$ norm of $\qT$.

The above theorem is proved in Section~\ref{Section:CC}. To show
(\ref{Lq-inv}) we will borrow techniques developed in the modern
treatment of nonlinear PDEs in critical cases: induction on energy
and profile decompositions, in order to deal with the lack of
compactness (\cite{KM06},\cite{G98}, \cite{BG} ). The novelty here
is that these ideas will be used in a nonstandard fashion and on
the static equations (\ref{eq_n}), rather than on the nonlinear
flow (\ref{DSIIdefoc}).

In Section \ref{Section:Scattering} we will show how to use the above result
in order to construct $m_\pm$ and $\mathcal Sq$ assuming only $\qT\in L^2(\R^2)$.

The Scattering Transform can be viewed as a nonlinear Fourier
transform, and it shares many of the same properties.  The
linearization of $\mathcal S$ at $\qT=0$ is essentially the Fourier
transform:
\begin{align}\label{S_Four}
\mathcal S\qT(k) = \ol{\hat\qT(k)} + \mathcal O(\qT^2)
\end{align}
where
\begin{align}\label{Four}
\hat\qT(k) = \frac{i}{\pi}\int_{\R^2}e_{-k}(z)\qT(z)dz.
\end{align}
We will use the normalization (\ref{Four}) for the Fourier
transform throughout the paper (except in Section
\ref{Section:Estimates}).

Writing $\tT:=\mathcal S\qT$, and setting
\begin{align}
n_\pm := \frac{1}{2}\Big((m_+ + m_-)\pm e_{-k}\ol{(m_+ -
m_-)}\,\Big),
\end{align}
it turns out that the functions $n_\pm(z,k)$ solve equations in
$k$ which are the same as those solved by $m_\pm(z,k)$ in $z$,
with $\qT(z)$ replaced by $\tT(k)$:
\begin{align}\label{eq_m}
\dkb n_\pm = \pm e_{-k}\tT\ol{n_\pm}
\end{align}
with $n_\pm(z,k)\rightarrow 1$ as $|k|\rightarrow\infty$. The
Inverse Scattering transform of $\tT$ is then defined as
\begin{align}\label{ItT}
\mathcal I\tT(z) = \frac{1}{2\pi i}\int_{\R^2} e_k \ol{\tT(k)}
\Big(n_+(z,k) + n_-(z,k)\Big) dk.
\end{align}
Note that $n_++n_-=m_++m_-$ and under appropriate conditions on
$\qT$, one can show that $\qT=\mathcal I(\tT)$. Thus, with the
above notation conventions, the scattering transform is an
involution $\mathcal S^2=I$.

If we now evolve the potential $\qT$ according to the DSII
equation (\ref{DSIIdefoc}), the corresponding scattering data evolves (as was shown in \cite{BC}; see also \cite{Per}) according to:
\begin{align}\label{tT_time}
\frac{\partial}{\partial t}\tT(t,k) &= 2i(k^2+\ol k^2)\tT(t,k).
\end{align}
Thus, the Cauchy problem for the nonlinear equation
(\ref{DSIIdefoc}) may be solved in a manner analogous to the use
of the Fourier transform for linear PDEs, by performing
forward-scattering on the initial data $\qT_0$ then evolving the
scattering data linearly in time according to (\ref{tT_time}) and
then performing Inverse Scattering to determine $\qT$ at time $t$,
namely
\begin{align}\label{ST}
\begin{cases}
\tT_0(k) &= \mathcal{S}\qT_0(k)\\
\tT(t,k) &= e^{2i(k^2+\ol k^2)t}\tT_0(k)\\
\qT(t,z) &= \mathcal I\big(\tT(t,k)\big)(z).
\end{cases}
\end{align}

This Inverse Scattering approach to the solution of the DSII
equations dates back to Ablowitz and Fokas (\cite{AF82},
\cite{AF83} and \cite{AF84}) and Beals and Coifman (\cite{BC85},
\cite{BC} and \cite{BC89}). Beals and Coifman showed that for
initial data in the Schwartz class, (\ref{eq_n}) and (\ref{eq_m}) are
solvable, and the corresponding scattering data is also in the 
Schwartz class. They also proved that for potentials in the 
Schwartz class, the scattering transform satisfies the nonlinear Plancherel identity
\begin{align}\label{BC_Plan}
\int |\qT(z)|^2dz = \int|\tT(k)|^2dk,
\end{align}
and is a symplectomorphism.

Sung (\cite{SunI}, \cite{SunII}, \cite{SunIII}) carried out the
analysis of the scattering transform and its inverse to solve the
defocusing DSII for initial data $\qT_0\in L^2\cap L^p$ for some
$p\in [1,2)$ with $\hat\qT_0\in L^1\cap L^\infty$. Brown and
Uhlmann \cite{BU} proved that for $\qT\in L^p_c$ where $p>2$, the
scattering data $\tT\in L^2$. Tamasan \cite{Tam} proved that for
$\qT\in W^{\varepsilon,p}_c$, where $\varepsilon>0$ and $p>2$, the
scattering data $\tT\in L^r$ for each $r > 2/(\varepsilon + 1)$.
Brown \cite{Bro} proved the Plancherel identity and Lipschitz
continuity of the scattering transform for $\qT\in L^2$ of
sufficiently small norm. Brown estimated directly the series
expansion of $\tT$ in multi-linear terms in $\qT$ (see also
\cite{ZB11} for such estimates). He stated as open questions
whether one can remove the smallness assumption and whether
solutions to (\ref{DSIIdefoc}) can be constructed when $\qT$ is in
$L^2$. We will address these questions in this paper.

There has been significant recent progress on the problem of the
validity of the Plancherel identity (\ref{BC_Plan}) without a
smallness assumption. Perry \cite{Per} proved that for $\qT$ in
the weighted Sobolev space $H^{1,1}$ the scattering data $\tT\in
H^{1,1}$. In addition, he proved local Lipschitz continuity of the
map $\mathcal{S}:H^{1,1}\rightarrow H^{1,1}$. He used these
results to show global well-posedness for defocusing DSII for
initial data in $H^{1,1}$. Astala, Faraco and Rogers \cite{AFR}
sharpened part of Perry's proof to show local Lipschitz continuity
of the scattering map $\mathcal{S}$ from $H^{s,s}$ to $L^2$ for $s
\in (0, 1)$ thus extending the Plancherel identity to this space.
Perry, Otto and Brown \cite{BOP} then showed that the scattering
transform maps $\qT\in H^{\alpha,\beta}$ to $\tT\in
H^{\beta,\alpha}$ for $\alpha,\beta>0$ thus establishing further
precise analogy between the properties of the scattering transform
and the Fourier transform.

In this paper we prove the Plancherel theorem for the Scattering
Transform for general $\qT$ in $L^2 (\R^2)$. To do so, we need new
bounds on  $\Icb$ (or, more generally, fractional integrals),
which, in the presence of an oscillatory term (see (\ref{eq_n}))
allow us to capture the behaviour of the functions $m_\pm(z,k)\rightarrow 1$ as $|k|\rightarrow\infty$
without assuming any smoothness on $\qT$. As well, in order to
make sense of the formula (\ref{tT}), we will need a new result on
the $L^2$-boundedness of pseudo-differential operators with
non-smooth symbols. In Section \ref{Section:Estimates}, we give
proofs of these bounds valid in any dimension, as they may be of
independent interest.

We are now ready to state precisely our Plancherel theorem.

\begin{theorem}\label{ThmScatteringTransform}
The nonlinear Scattering Transform $\mathcal{S}:\qT\mapsto \tT$ is
a $C^1$ diffeomorphism ${\mathcal{S}}: L^2(\R^2)\rightarrow L^2(\R^2)$,
satisfying:
\newline
\begin{enumerate}
\item  The Plancherel Identity:
\begin{align}\label{t1-l2}
\|{\mathcal{S}}\qT\|_{L^2} = \|\qT\|_{L^2}.
\end{align}
\item  The pointwise bound:
\begin{align}\label{t1-point}
&|{\mathcal{S}}\qT(k)|\leq C(\|\qT\|_{L^2})M\hat \qT(k)
\end{align}
for a.e. $k$, where M denotes the Hardy-Littlewood Maximal
function.
\item Locally uniform bi-Lipschitz continuity:
\begin{align}\label{t1-lip}
\frac{1}{C}\|{\mathcal S}\qT_1 - {\mathcal S}\qT_2\|_{L^2}\leq
\|\qT_1 - \qT_2\|_{L^2}\leq C\|{\mathcal S}\qT_1 - {\mathcal
S}\qT_2\|_{L^2}
\end{align}
where
\begin{align*}
C = C(\|\qT_1\|_{L^2})C( \|\qT_2\|_{L^2}).
\end{align*}
\item Bound on the derivative:
\begin{align}\label{t1-diff}
\Big\| \frac{\delta\mathcal S}{\delta \qT} \Big\|_{L^2 \to L^2} \leq C(\|\qT\|_{L^2}).
\end{align}
\item Inversion Theorem: \begin{align*}
\mathcal S^{-1}=\mathcal S.
\end{align*}
\item$\mathcal S$ is a symplectomorphism \footnote{We are grateful to the anonymous referee who suggested that we should also prove this additional property for the Scattering Transform.}: for every $\qT, \qT_1, \qT_2\in L^2(\R^2)$
\begin{align}
\omega_2(\frac{\delta\mathcal S}{\delta \qT}\Big|_{\qT}\qT_1,\frac{\delta\mathcal S}{\delta \qT}\Big|_{\qT}\qT_2) = \omega_1(\qT_1,\qT_2),
\end{align}
where $\omega_1$, $\omega_2$ are the symplectic forms
\begin{align*}
\omega_1(\qT_1,\qT_2) = -\Im \int\qT_1(z)\ol {\qT_2(z)}dz, \quad \omega_2(t_1,t_2) = -\Im \int\ol{t_1(k)}t_2(k)dk.
\end{align*}
\end{enumerate}
\end{theorem}
We will in fact prove an identity more general than (\ref{t1-l2}) (see Corollary \ref{cor_Plan}) which shows to what extent $S$ departs from being an isometry. Furthermore, explicit formulas for the derivative $\frac{\delta\mathcal S}{\delta \qT}$ and its inverse are given in Lemma \ref{LemDiffeom}.

As a consequence of properties (1), (2) and (5) above, we note the
following pointwise bound on $\qT$ in terms of the Fourier
transform of its scattering transform.
\begin{corollary}\label{CorBound}
If $\qT \in L^2(\R^2)$ and $\tT= \mathcal S(\qT)$ then for a.e.
$z$ we have:
\begin{align*}
&|\qT(z)|\leq C(\|\qT\|_{L^2})M\hat \tT(z).
\end{align*}
\end{corollary}

\bigskip

\subsection{Global Well-Posedness for the Defocusing DSII Problem}
One immediate application of Theorem \ref{ThmScatteringTransform}
will be to show global well-posedness of the Cauchy problem for
the defocusing Davey-Stewartson equation for arbitrary initial
data in $L^2(\R^2)$. In particular, the above Corollary will yield
pointwise control of the solution to DSII by the maximal function
of a solution of the linear flow and thus will allow us to
transfer Strichartz estimates on the linearization of
(\ref{DSIIdefoc}) to bounds on the nonlinear flow.

We recall the formulation of (\ref{DSIIdefoc}) as an integral
equation (\cite{GuiSau}). The Cauchy problem (\ref{DSIIdefoc}) has
a corresponding linear flow
\begin{align}\label{LinDSII}
\begin{cases}
&i\partial_t \tilde\qT + 2(\dzb^2+\dz^2)\tilde\qT = 0\\
&\tilde\qT(0,z)=\qT_0(z).
\end{cases}
\end{align}
Let $U(t)$ be the solution operator to the linear problem
(\ref{LinDSII})
\begin{align*}
U(t)\qT_0:=\tilde\qT(t,\cdot)= e^{2it(\dz^2+\dzb^2)}\qT_0.
\end{align*}
Using Duhamel's principle, (\ref{DSIIdefoc}) can be written as the
following nonlinear integral equation for $\qT(t)=:\qT(t,\cdot)$
\begin{align}\label{Duhamel}
\qT(t) = U(t)\qT_0+\Lambda(\qT)(t)
\end{align}
with
\begin{align}\label{DuhamelInt}
\Lambda(\qT)(t) = i\int_0^t
U(t-s)\Big(\qT(s)(\dzb\Ic+\dz\Icb)|\qT(s)|^2\Big)ds.
\end{align}
Ghidaglia and Saut~\cite{GuiSau} proved that for any $\qT_0\in L^2(\C)$ the
problem (\ref{Duhamel}) has a unique solution in the Strichartz type space
\begin{align*}
X_T := C([0,T],L^2_z(\C))\cap L^4_{t,z}([0,T]\times\C)
\end{align*}
for some $T$ which depends on $\qT_0$; they also showed that for
$\qT_0$ with sufficiently small $L^2$ norm this holds for all $T$.
Using the Inverse Scattering method, Perry proved global well
posedness for general initial data $\qT_0\in H^{1,1}$. Our
Plancherel Theorem yields the following:

\begin{theorem} \label{ThmDSII}(Global well-posedness for defocusing DSII on $L^2$)
Given $\qT_0\in L^2$, there exists a unique solution to the Cauchy
problem (\ref{DSIIdefoc}) in the sense of equation (\ref{Duhamel})
such that:
\begin{enumerate}
\item Regularity:
\[
\qT(t,z)\in C(\R, L^2_z(\C))\cap L^4_{t,z}(\R\times\C).
\]
\item Uniform bounds: conservation of mass $\|\qT(t,\cdot)\|_{L^2} = \|\qT_0\|_{L^2}$ for all $t\in \R$  and
\[
\int_\R\int_{\R^2}|\qT(t,z)|^4dzdt \leq
C(\|\qT_0\|_{L^2}).
\]
\item Pointwise bound:
\[
|\qT(t,z)| \leq C(\|\qT_0\|_{L^2})M\qT^{\text{lin}}(t,z)
\]
where
\[
\qT^{\text{lin}}(t,\cdot) = U(t)\ol{\widehat{\mathcal S\qT_0}}.
\]
\item Stability:  if  $\qT_1(t,\cdot)$ and $\qT_2(t,\cdot)$ are
two solutions corresponding to initial data $\qT_1(0,\cdot)$  and
$\qT_2(0,\cdot)$ with $\|\qT_j(0,\cdot)\|_{L^2}\leq R$ then
\[
\|\qT_1(t, \cdot) -\qT_2(t,
\cdot)\|_{L^2} \leq C(R) \|\qT_1(0,\cdot) - \qT_2(0, \cdot)\|_{L^2} \qquad \text{for all $t\in\R$.}
\]
\end{enumerate}
\end{theorem}

We remark that much of the conclusion of this theorem closely
resembles the conclusion of Dodson's result \cite{D16} for the two
dimensional cubic defocusing NLS problem
\[
i u_t + \Delta u = u|u|^2,
\]
(see also the prior work \cite{KTV09}). Written in a similar format, the DSII problem
has the form
\[
i q_t + (\partial_{1}^2 -\partial_2^2) q = q L (|q|^2), \qquad
L(D) = \frac{D_1^2 -D_2^2}{D_1^2+D_2^2}.
\]
Whereas the small data theory for the two problems is completely
similar from a dispersive stand-point (i.e. perturbative, based
on Strichartz estimates), the large-data approach in the present,
completely integrable case and in Dodson's work are completely
different. The large data problem for the other, non-integrable
cases in the same DS family remains open at present. This includes
for instance the problem
\[
i q_t + (\partial_{1}^2 -\partial_2^2) q = q  |q|^2.
\]

The next theorem provides one more convincing motivation for the
study of the Scattering Transform, if one seeks to understand the
large-time behaviour of the solutions to the DSII equation. We
first recall the definition of the wave operators, in the sense of
nonlinear scattering theory.
\begin{definition} Let $\qT_0\in L^2(\R^2)$ and let $\qT(t,z)$ be the
solution to the Cauchy problem (\ref{DSIIdefoc}). Define $W_+\qT_0
= \qT_+$ if there exists a unique $\qT_+\in L^2(\R^2)$ such that
\begin{align*}
\lim_{t\rightarrow\infty}\|\qT(t,\cdot)-U(t)\qT_+\|_{L^2(\R^2)}=0.
\end{align*}
Similarly $W_-\qT_0 = \qT_-$ if
\begin{align*}
\lim_{t\rightarrow-\infty}\|\qT(t,\cdot)-U(t)\qT_-\|_{L^2(\R^2)}=0.
\end{align*}
\end{definition}
We can now state the following further consequence of the
Plancherel theorem:

\begin{theorem}\label{ThmScattering} (Wave operators and asymptotic completeness for defocusing
DSII)\\
\noindent a) The Wave operators $W_\pm$ for the defocusing DSII
equation are well defined on every $\qT_0\in L^2(\R^2)$ and
\begin{align*}
W_\pm\qT_0 = \ol{\widehat{\mathcal S \qT_0}}.
\end{align*}
\noindent b) The Wave operators $W_\pm$ are surjective, in fact
norm-preserving diffeomorphisms of $L^2$.
\end{theorem}

Perry \cite{Per} established the same large-time asymptotic
behaviour in the $L^\infty$ norm, for initial data in $H^{1,1}
\cap L^1$. Kiselev (\cite{Kis97}, \cite{Kis04}) had similar
results under more restrictive assumptions.

An interesting consequence of Theorem \ref{ThmScattering} is that
the temporal scattering operator $W_+(W_-)^{-1}$ for the
defocusing DSII equation (i.e. the operator which sends $q_-$ to
$q_+$) is equal to the identity.

\bigskip

\subsection{Application to Inverse Boundary Value Problems}
We next discuss the application of our Plancherel theorem to the
Inverse Boundary Value Problem of Calder\'on. Let $\Omega$ be a bounded 
simply connected domain in $\R^2\simeq\C$ with $C^{1,1}$ boundary
$\partial\Omega$. We denote by $\nu$ the outer unit normal to
$\partial\Omega$ and $\tau$ the unit tangent in the counter-clockwise direction.  Consider the Dirichlet problem
\begin{align} \label{bvp}
\begin{cases}
&\nabla\cdot(\sigma\nabla u)=0 \text{ in }\Omega\\
&u\Big|_{\partial\Omega} = g.
\end{cases}
\end{align}
The Dirichlet-to-Neumann map is defined as
\begin{align}
\Lambda_\sigma g:=\sigma \frac{\partial
u}{\partial\nu}\Big|_{\partial\Omega},
\end{align}
with $u$ the solution to (\ref{bvp}). The function $\sigma$ models
the inhomogeneous conductivity of $\Omega$, and $\Lambda_\sigma$
represents the information observable by voltage and current
measurements at the boundary. Calder\'on posed the problem of
establishing whether $\sigma$ is uniquely determined by
$\Lambda_\sigma$ and, if so, of finding a way to calculate
$\sigma$ from knowledge of $\Lambda_\sigma$.

There is by now an extensive literature on this and related
problems. See for instance \cite{ALP} for a recent review. We only
briefly recall some of the pertinent results. The first global
uniqueness theorem was proved by Sylvester-Uhlmann \cite{SU} for
smooth conductivities in dimensions 3 or higher. A reconstruction
method was given in \cite{Nac88}. In three dimensions or higher,
uniqueness has been shown for Lipschitz conductivities close to
the identity in \cite{HT}; the smallness condition was removed in
\cite{CR}. In dimensions $n=3,4$ Haberman \cite{Hab} has proved
uniqueness for conductivities in $W^{1,n}(\Omega)$.

In two dimensions, the first global uniqueness and reconstruction
result was obtained in \cite{Nac} for conductivities in
$W^{2,p}(\Omega)$ with $p>1$, by connecting $\Lambda_\sigma$ to a
scattering transform for a Schr\"odinger equation. This was
refined to $W^{1,p}(\Omega)$ with $p>2$ in \cite{BU} using the
scattering transform studied in this paper. In \cite{AP}, Astala
and P\"aiv\"arinta succeeded in proving uniqueness for general
$L^\infty$ conductivities bounded below. In \cite{ALP}, uniqueness
is extended to a larger class of conductivities that allow some
$\sigma$ which need not be bounded from above or below. In a
recent paper, C\^{a}rstea and Wang \cite{CW} have shown uniqueness for
conductivities $\sigma$ in $W^{1,2}(\Omega)$ which are bounded
from below assuming $\|\nabla\log\sigma\|_{L^2}$ sufficiently
small.

Here we use Theorem \ref{ThmScatteringTransform} to prove global
uniqueness for conductivities $\sigma>0$ a.e. with the property
that
\begin{equation}\label{sigma}
\log \sigma \in \dot H^1(\Omega),  \qquad \sigma = 1 \text{ on
}\partial \Omega.
\end{equation}
This is in line with the sharpest results known in higher
dimensions, mentioned above (\cite{Hab}). Notably we do not assume
any $L^\infty$ type bounds on $\sigma$ from above or below \footnote{However, for any $1 \leq p < \infty$, we do have $\sigma$ and $\sigma^{-1}$ in $L^p$, as follows from the Poincar\'e inequality applied to $\log\sigma$ and Theorem 7.21 in \cite{GT}.}.
As examples, consider the conductivities $\sigma_\alpha(x) = (-\log|x|)^\alpha$, for any $\alpha\in \R$, on the domain $\Omega = \{x: |x|<e^{-1}\}$. We have $|\nabla\log \sigma_\alpha(x)|=-|\alpha|/|x|\log |x|\in L^2(\Omega)$, hence $\log \sigma_\alpha\in \dot H^1(\Omega)$. For $\alpha<0$ these conductivities degenerate at the origin, while for $\alpha>0$ they are  unbounded; as well, for large $\alpha$ they are not covered by the uniqueness results in \cite{ALP} (Theorem 1.9) or \cite{CW}. We first need to make sure that the Dirichlet
problem (\ref{bvp}) is solvable.

\begin{theorem}\label{t:bvp-solve}
  Assume that $\sigma$ is as in \eqref{sigma}.  Then for every $g \in
  H^1(\partial\Omega)$ there exists a unique solution $u$ to the Dirichlet
  problem \eqref{bvp} with $ \sigma^\frac12 \nabla u \in H^\frac12(\Omega)$. Furthermore,
  $\partial u /\partial \nu \in L^2(\partial \Omega)$.
\end{theorem}
In particular this insures that $\Lambda_\sigma$ is a well-defined operator
\[
\Lambda_\sigma: H^1(\partial \Omega) \to L^2(\partial \Omega).
\]
Now we can state our main result on the  Calder\'on problem:

\begin{theorem} \label{reconst}
Assume the conductivity $\sigma>0$ is such that $\log\sigma \in
\dot H^1$. We also assume, for simplicity, that $\sigma=1$ on
$\partial\Omega$. Then we can reconstruct $\sigma$ from knowledge
of $\Lambda_\sigma$.
\end{theorem}

We will obtain Theorems \ref{t:bvp-solve} and \ref{reconst} as
consequences of corresponding results for pseudo-analytic
functions, which are also of interest. 

More precisely, a standard computation shows that if $u$ is a real-valued solution for $(\ref{bvp})$ then the function $v = \sigma^\frac12 \partial u$ solves the equation
\begin{align}\label{dbar-model-eq}
\dbar v -  q \bar v = 0  \text{ in } \Omega
\end{align}
with
\begin{equation}\label{q-def}
q = -\frac12 \partial \log \sigma.
\end{equation}
Moreover, on the boundary $\partial\Omega$ we have (using (\ref{dzbdz}) and the assumption  $\sigma=1$ on
$\partial\Omega$):
\begin{align}\label{du_dnu}
\frac{\partial u}{\partial \nu} = 2 \Re(\nu\partial u) = 2\Re (\nu v)
\end{align}
and
\begin{align}\label{du_dtau}
\frac{\partial u}{\partial \tau} = -2 \Im(\nu\partial u) = -2\Im (\nu v),
\end{align}
where we interpret the outer normal also as a complex-valued function $\nu = \nu_1 + i\nu_2$ on $\partial\Omega$. In particular, given $g = u|_{\partial\Omega}$ we can determine $\Im (\nu v) = -\half \frac{\partial g}{\partial \tau}$ on $\partial \Omega$. We are thus led to study the following boundary value problem of pseudo-analytic function $v$: 

\begin{align}\label{dbar-model}
\begin{cases}
\dbar v -  q \bar v = 0  & \text{in } \Omega\\
\Im (\nu v) =  g_0  & \text{on } \partial \Omega,
\end{cases}
\end{align}
for $g_0\in L^2(\partial\Omega)$ with integral zero, and to define an associated Hilbert Transform type operator on $\partial \Omega$ as:
\begin{equation}
\Hq g_0 :=   \Re (\nu v).
\end{equation}
For then, in view of (\ref{du_dnu}) and (\ref{du_dtau}), we will  have the following relation:
\begin{align}\label{DtN_H}
\Lambda_\sigma = -\mathcal H_\qT\frac{\partial}{\partial \tau},
\end{align}
so that the Dirichlet-to-Neumann map $\Lambda_\sigma$ for $(\ref{bvp})$ will determine the boundary operator $\mathcal H_\qT$ for (\ref{dbar-model}). In turn, $\mathcal H_\qT$ will be shown to determine the Scattering Transform of $\qT$ (extended to be zero outside $\Omega$) thus allowing the use of Theorem \ref{ThmScatteringTransform} (5) to complete the solution of the inverse problem. We will first prove the result on the solvability of the forward problem (\ref{dbar-model}).

\begin{theorem}\label{dbar-solve}
Assume that $q \in L^2$ is given by \eqref{q-def} with $\sigma$ as
in \eqref{sigma}. Then for each real-valued $g_0 \in L^2(\partial
\Omega)$ with integral zero the problem \eqref{dbar-model} admits
a unique solution  $v \in H^{\frac12}(\Omega)$. Furthermore,  $v
\in L^2(\partial \Omega)$ and
\begin{equation}
\| v \|_{  H^{\frac12}(\Omega)} + \| v\|_{L^2(\partial \Omega) }
\leq C(q)  \| g_0\|_{L^2(\partial \Omega)}.
\end{equation}
\end{theorem}
Thus $\mathcal H_\qT$ is well-defined as a bounded operator on $L^2(\partial\Omega)$:
\begin{equation}
L^2    \ni     \Im (\nu v) = g_0 \to    \Hq g_0 :=   \Re (\nu
v)\in L^2.
\end{equation}
Our main reconstruction theorem for \eqref{dbar-model} states that
one can recover $\qT$ from this boundary operator. One may consider it
as analogous to the result in \cite{AP} where the Hilbert
transform for a Beltrami equation is shown to determine the
corresponding Beltrami coefficient.

\begin{theorem}\label{dbar-inverse}
Assume that $q \in L^2$ is given by \eqref{q-def} with $\sigma$ as
in \eqref{sigma}.  Then we can reconstruct $q$ from knowledge of
$\Hq$.
\end{theorem}

We will in effect consider these last two theorems as the main
ones, with the results for the  Calder\'on problem as
straightforward consequences.
\bigskip

{\bf Acknowledgements:}

The authors would like to thank Alexandru Tamasan for many helpful
discussions at the early stages of investigation.  A. Nachman and
D. Tataru are grateful to IHP for hospitality and support during
the program on Inverse Problems in 2015, which allowed us to
initiate this project. The authors are also grateful to Xian Liao, Peter Perry, Mihai Tohaneanu, Pavel Zorin-Kranich and the anonymous referees for carefully reading the manuscript and helping us correct a number 
of typos and inaccuracies.

D. Tataru was partially supported by the NSF grant DMS-1266182 as
well as by the Simons Investigator grant from the Simons
Foundation. A. Nachman was partially supported by the  NSERC
Discovery Grant RGPIN-06329.

%%% ======================================================================
\section{Estimates on Fractional Integrals and
Pseudo-differential Operators}\label{Section:Estimates}
%%% ======================================================================

This section is devoted to the proofs of new boundedness theorems
on fractional integrals, pointwise multipliers in negative Besov
spaces and pseudo-differential operators with non-smooth symbols.
These results will be crucial in the rest of the paper. The proofs
in this section are valid in all dimensions.

Recall that the Hardy-Littlewood Maximal function is defined for
locally integrable functions $f:\R^n\rightarrow\C$ as:
\begin{align*}
Mf(x) = \sup_{r>0}\frac{1}{|B(x,r)|}\int_{B(x,r)}|f(y)|dy
\end{align*}
and yields a bounded operator on $L^p$ for $1<p \leq \infty$ (see,
for instance \cite{Ste}). Also recall the mixed $L^p$ norm:
\begin{align*}
\|f\|_{L^q_y L^p_x} = \Big(\int\Big(\int|f(x,y)|^p
dx\Big)^\frac{q}{p} dy\Big)^\frac{1}{q}.
\end{align*}

We have the following pointwise bound on fractional integrals:
\begin{theorem}\label{FractionalIntegral}
For $0<\alpha<n$, $f\in L^p(\R^n)$, $1\leq p\leq 2$
\begin{align*}
&\text{a)}\enspace\enspace\big|(-\Delta)^{-\frac{\alpha}{2}}f(x)\big|\leq
c_{n,\alpha}\Big( \lambda^{n-\alpha}M\hat
f(0)+\lambda^{-\alpha}Mf(x)\Big)\enspace\enspace\text{ for any
$\lambda>0$}\\
&\text{b)}\enspace\enspace\big|(-\Delta)^{-\frac{\alpha}{2}}f(x)\big|\leq
c_{n,\alpha} \Big(M\hat
f(0)\Big)^{\frac{\alpha}{n}}\Big(Mf(x)\Big)^{1-\frac{\alpha}{n}}.
\end{align*}
\end{theorem}

\begin{proof}
First we note that the restriction on $f$ along with the Hausdorff-Young inequality assures that  $f$ and $\hat f$ are locally integrable and so $Mf$ and $M\hat f$ are well defined. To simplify notation, we will use $\lesssim$ in place of $\leq
c_{n,\alpha}$. Using the Littlewood-Paley decomposition, we write
\begin{align}\label{PL}
(-\Delta)^{-\frac{\alpha}{2}}f(x)=\frac{1}{(2\pi)^n}\sum_{j=-\infty}^\infty\int_{\R^n}\psi_j(\xi)\frac{e^{ix\cdot\xi}}{|\xi|^\alpha}\hat
f(\xi)d\xi
\end{align}
with $\psi_j(\xi)=\psi(\xi/2^j)$ supported in
$2^{j-1}<|\xi|<2^{j+1}$. Fix $j_0$, for now. We estimate the terms
in (\ref{PL}) with $j\leq j_0$ using
\begin{align*}
\int_{|\xi|<r}|\hat f(\xi)|d\xi\leq c_n r^n M\hat f(0):
\end{align*}
\begin{align}\label{PL_2}
\sum_{j=-\infty}^{j_0}\int_{\R^n}\frac{\psi_j(\xi)}{|\xi|^\alpha}
|\hat f(\xi)|d\xi&\lesssim \sum_{j=-\infty}^{j_0}2^{-j\alpha}
M\hat
f(0)2^{jn}\\
\nonumber &\lesssim 2^{j_0(n-\alpha)}M\hat f(0),
\end{align}
since $\alpha<n$. We bound the terms in (\ref{PL}) with $j\geq
j_0$ by
\begin{align*}
\sum_{j=j_0}^{\infty}\int_{\R^n}|K_j(y)||f(x-y)|dy,
\end{align*}
with
\begin{align*}
K_j(y)=\frac{1}{(2\pi)^n}\int_{\R^n}\psi_j(\xi)\frac{e^{iy\cdot\xi}}{|\xi|^\alpha}d\xi.
\end{align*}
The integral kernel $K_j$ can be estimated by
\begin{align}\label{K_bound}
|K_j(y)|\lesssim|y|^{-N}2^{j(n-\alpha-N)}
\end{align}
for any integer $N\geq 0$. This estimate is obtained, as usual, by
writing
\begin{align*}
K_j(y)=\frac{1}{(i|y|^2)^N}\frac{1}{(2\pi)^n}\int_{\R^n}\frac{\psi(\xi/2^j)}{|\xi|^\alpha}(y\cdot\nabla_\xi)^N
e^{i y\cdot \xi}d\xi.
\end{align*}
and integrating by parts $N$ times. We write
\begin{align}\label{ineq_1}
\int_{|y|\geq 2^{-j}}|K_j(y)||f(x-y)|dy \nonumber
&=\sum_{l=-j}^{\infty}\int_{2^{l}\leq|y|\leq
2^{l+1}}|K_j(y)||f(x-y)|dy\\
\nonumber
&\lesssim\sum_{l=-j}^{\infty}2^{j(n-\alpha-N)}2^{-lN}\int_{2^{l}\leq|y|\leq
2^{l+1}}|f(x-y)|dy\\
\nonumber
&\text{(using (\ref{K_bound}) with $N>n$)}\\
\nonumber
&\lesssim \sum_{l=-j}^{\infty}2^{j(n-\alpha-N)}2^{-lN}2^{(l+1)n}Mf(x)\\
\nonumber
&\lesssim 2^{j(n-\alpha-N)}2^{n}2^{-j(n-N)}Mf(x)\\
&\lesssim 2^{-j\alpha}Mf(x).
\end{align}
For $|y|<2^{-j}$ we use (\ref{K_bound}) with $N=0$:
\begin{align}\label{ineq_2}
\int_{|y|\leq 2^{-j}}|K_j(y)||f(x-y)|dy \nonumber &\lesssim
2^{j(n-\alpha)}\int_{|y|\leq
2^{-j}}|f(x-y)|dy\\
\nonumber
&\lesssim 2^{j(n-\alpha)}2^{-jn}Mf(x)\\
&= 2^{-j\alpha}Mf(x).
\end{align}
The inequalities (\ref{ineq_1}) and (\ref{ineq_2}) yield
\begin{align*}
\sum_{j=j_0}^{\infty}\int_{\R^n}|K_j(y)||f(x-y)|dy\lesssim
Mf(x)\sum_{j=j_0}^{\infty}2^{-j\alpha}\lesssim 2^{-j_0\alpha}Mf(x).
\end{align*}
Returning to (\ref{PL}) and also using (\ref{PL_2}), we obtain:
\begin{align*}
\big|(-\Delta)^{-\frac{\alpha}{2}}f(x)\big|\lesssim
2^{j_0(n-\alpha)}M\hat f(0)+2^{-j_0\alpha}Mf(x)
\end{align*}
for any $j_0$. This proves inequality a). Inequality b) then
follows by optimizing over $\lambda$.
\end{proof}

We state explicitly the special case of the above in the form
which will be used in subsequent sections. These estimates will allow
us to obtain precise control of $m(\cdot,k)$ and $\tT(k)$ for
large $k$ without any smoothness assumptions on $\qT$.

\begin{corollary}\label{FI_}
For $\qT\in L^2(\C)$
\begin{align*}
\text{a) }&|\Icb(e_{-k}\qT)(x)|\lesssim \Big(M\hat
\qT(k)\Big)^{\half}\Big(M\qT(x)\Big)^\half\\
\text{b) }&\|\Icb(e_{-k}\qT)\|_{L^4}\lesssim
\|\qT\|_{L^2}^\half\Big(M\hat
\qT(k)\Big)^{\half}.\\
\end{align*}
\end{corollary}

\begin{proof}
Assertion a) follows directly from Theorem \ref{FractionalIntegral} b) with $\alpha=1$, $n=2$ and $p=2$. Assertion b) follows
from assertion a) and the boundedness of $M$ on $L^2$.
\end{proof}

We next use Theorem \ref{FractionalIntegral} to prove $L^2$
boundedness for a class of pseudo-differential operators with
non-smooth symbols (See the monograph \cite{CM} for an extensive
investigation of such problems). The result we need here does not
appear to be available in the literature. It will allow us to show
that the scattering transform is well defined and in $L^2$ as a
function of $k$.

\begin{theorem}\label{ThmPsiDO}
Let $0\leq\alpha<n$. Suppose $a(x,\xi)$ satisfies\footnote{Assuming that $a$ decays in $\xi$ at infinity, the condition (i) follows from (ii) by Sobolev embeddings, but is written separately for reference purposes.}
\begin{align*}
i)\enspace &\int_{\R^n} \int_{\R^n}\big|
a(x,\xi)\big|^{\frac{2n}{n-\alpha}}d\xi
dx<\infty\enspace\enspace\text{ and}\\
ii)\enspace
&\|(-\Delta_\xi)^{\frac{\alpha}{2}}a(x,\xi)\|_{L^{\frac{2n}{n+\alpha}}_\xi}\in
{L^{\frac{2n}{n-\alpha}}_x}.
\end{align*}
Then the pseudo-differential operator
\begin{align}\label{PseudoDiff}
a(x,D)f(x) := \frac{1}{(2\pi)^n}\int_{\R^n}
e^{ix\cdot\xi}a(x,\xi)\hat f(\xi)d\xi
\end{align}
is bounded on $L^2$ with
\begin{align}\label{PseudoDiff_L2}
\|a(x,D)f\|_{L^2} \leq
c_{\alpha,n}\|f\|_{L^2}\|(-\Delta_\xi)^{\frac{\alpha}{2}}a(x,\xi)\|_{L^{\frac{2n}{n-\alpha}}_x
L^{\frac{2n}{n+\alpha}}_\xi }.
\end{align}
Moreover, we have the pointwise bound
\begin{align}\label{PointwiseBound}
|a(x,D)f(x)| \leq
c_{\alpha,n}(Mf(x))^{\alpha/n}\|(-\Delta_\xi)^{\frac{\alpha}{2}}a(x,\cdot)\|_{L^{\frac{2n}{n+\alpha}}}\|f\|^{1-\frac{\alpha}{n}}_{L^2}
\end{align}
for a.e. x.
\end{theorem}

\begin{proof}
The case $\alpha = 0$ follows by Cauchy-Schwartz. To investigate the case $0<\alpha<n$, suppose first that $f$ is in Schwartz class. Let $b(x,\xi) =
(-\Delta_\xi)^{\frac{\alpha}{2}}a(x,\xi)$. Since $a\in
L^{\frac{2n}{n-\alpha}}$, then $a(x,\xi) =
(-\Delta_\xi)^{-\frac{\alpha}{2}}b(x,\xi)$ and we have
\begin{align*}
|a(x,D)f(x)| \leq \frac{1}{(2\pi)^n}\int_{\R^n}
\Big|(-\Delta_\xi)^{-\frac{\alpha}{2}}\Big(e^{ix\cdot\xi}\hat
f(\xi)\Big)\Big|\enspace\Big|b(x,\xi)\Big|d\xi.
\end{align*}
By Theorem \ref{FractionalIntegral}
\begin{align*}
\Big|(-\Delta_\xi)^{-\frac{\alpha}{2}}\Big(e^{ix\cdot\xi}\hat
f(\xi)\Big)\Big|\lesssim (Mf(x))^{\frac{\alpha}{n}}(M\hat
f(\xi))^{1-\frac{\alpha}{n}}.
\end{align*}
Hence,
\begin{align*}
|a(x,D)f(x)| &\lesssim (Mf(x))^{\frac{\alpha}{n}}\int_{\R^n}
(M\hat f(\xi))^{1-\frac{\alpha}{n}}\Big|b(x,\xi)\Big|d\xi\\
&\lesssim
(Mf(x))^{\frac{\alpha}{n}}\big\|b(x,\cdot)\big\|_{L^{\frac{2n}{n+\alpha}}}\big\|(M\hat
f)^{\frac{n-\alpha}{n}}\big\|_{L^{\frac{2n}{n-\alpha}}}\\
&\lesssim
(Mf(x))^{\frac{\alpha}{n}}\big\|b(x,\cdot)\big\|_{L^{\frac{2n}{n+\alpha}}}\big\|
f\|_{L^2}^{1-\frac{\alpha}{n}}
\end{align*}
for a.e. $x$. This proves (\ref{PointwiseBound}) for $f$ in
Schwartz class. We may then extend by continuity to $f\in L^2$.
Therefore, we have
\begin{align*}
\|a(x,D)f(x)\|_{L^2} &\lesssim
\|Mf\|_{L^2}^{\frac{\alpha}{n}}\|b\|_{L^{\frac{2n}{n-\alpha}}_x
L^{\frac{2n}{n+\alpha}}_\xi }
\|f\|_{L^2}^{1-\frac{\alpha}{n}}\\
&\lesssim
\|Mf\|_{L^2}^{\frac{\alpha}{n}}\|(-\Delta_\xi)^{\frac{\alpha}{2}}a\|_{L^{\frac{2n}{n-\alpha}}_x
L^{\frac{2n}{n+\alpha}}_\xi }
\|f\|_{L^2}^{1-\frac{\alpha}{n}}\\
&\lesssim \|(-\Delta_\xi)^{\frac{\alpha}{2}}a\|_{L^{\frac{2n}{n-\alpha}}_x
L^{\frac{2n}{n+\alpha}}_\xi }\|f\|_{L^2}.
\end{align*}
\end{proof}

We conclude this section with an estimate on pointwise
multipliers which will be needed in the proof of Theorem~\ref{ThmMain}
in Section~\ref{Section:CC}. First note that if $0\leq r< \frac{n}{2}$ and $\qT\in
L^{\frac{n}{2r}}$ then multiplication by $\qT$ yields a bounded
operator from the homogeneous Sobolev space $\dot H^r(\R^n)$ to
its dual $\dot H^{-r}(\R^n)$. (This follows easily from the
boundedness of the Sobolev embedding of $\dot H^r(\R^n)$ in
$L^{\frac{2n}{n-2r}}$). For the concentration compactness
arguments in Section \ref{Section:CC} we will need an extension of
this result to a larger space of potentials $\qT$, with negative
regularity index. Classes of pointwise multipliers between Sobolev
spaces have been extensively studied (see for example
\cite{MazShap}, \cite{LamGal} and further references given there).

We show that multiplication by any $\qT$ in the union of the
homogeneous Besov spaces $\dot B_\infty^{\frac{n}{p}-2r,p}$ with
$2\leq p< n/r$ yields a bounded operator from $\dot
H^{r}(\R^n)$ to $\dot H^{-r}(\R^n)$. We use the following notation
for the norm of the homogeneous Besov space:
\begin{align*}
\|f\|_{\dot B^{s,p}_q} = \Big(\sum_{k\in\Z}(2^{ks}\|P_k
f\|_{L^p})^q\Big)^{1/q}
\end{align*}
where $P_k$ are the Littlewood-Paley projections, $1\leq q,p\leq
\infty$ and $s\in \R$.

\begin{theorem}\label{ThmBilEst}
Let $0 <  r<n/2$ and $p\in[2,n/r)$. Then the following
bilinear estimate holds:
\begin{align}\label{BilEst}
\|\qT u\|_{\dot H^{-r}(\R^n)}\lesssim \|\qT\|_{\dot
B_\infty^{\frac{n}{p}-2r,p}(\R^n)}\|u\|_{\dot H^{r}(\R^n)}.
\end{align}
\end{theorem}

\begin{proof}
We use a dyadic Littlewood-Paley decomposition
\[
1 = \sum_{k \in \Z} P_k
\]
where  $P_k$ are  the standard dyadic  Littlewood-Paley operators, which are localized in the 
fre\-quency regions $A_k = \{\xi:2^{k-1}<|\xi|<2^{k+1}\}$. 
We will show that the dyadic components  $P_k(\qT u)$ of $\qT u$ satisfy the 
correct bound with off-diagonal decay,
\begin{equation}\label{dyadic}
\|P_k (\qT u)\|_{{\dot H^{-r}}} \lesssim \|q\|_{\dot B_\infty^{\frac{n}{p}-2r,p}}
\sum_{k''} 2^{-c |k-k''|} \| P_{k''} u \|_{\dot H^r}, \qquad c > 0
\end{equation}
This in turn easily implies \eqref{BilEst}.

To prove \eqref{dyadic} we write
\begin{align*}
P_k (\qT  u) = \sum_{(k',k'')\,\in\,\mathcal A_k}
P_k\big( P_{k'}\qT P_{k''}u \big),
\end{align*}
where  the sum is taken over the set
\begin{align*}
\mathcal A_k = \{(k',k'')\in\Z^2:A_k\cap(A_{k'}+A_{k''})\neq
\emptyset\}.
\end{align*}
Then we have 
\begin{align}\label{diadic}
\|P_k (\qT u)\|_{{\dot H^{-r}}} \lesssim
\sum_{(k',k'')\,\in\,\mathcal A_k} 2^{-rk}\|P_k(P_{k'}\qT
P_{k''}u)\|_{L^2}.
\end{align}
To estimate the terms in the above sum we use Bernstein inequality
applied to the Littlewood-Paley projections:
\begin{align}\label{Bern}
\|P_k f\|_{L^t}\leq 2^{kn(1/s-1/t)}\|P_kf\|_{L^s} 
\end{align}
for $1\leq s\leq t\leq \infty$.  We consider the three cases in the Littlewood-Paley trichotomy: 
\medskip

\textbf{ (i) Low-high interactions, $|k'' - k| \leq 2$, $k' \leq k+2$.}
Here we estimate
\[
\| P_{k'} \qT \|_{L^\infty} \lesssim  2^{2r k'}   \|q\|_{\dot B_\infty^{\frac{n}{p}-2r,p}}
\]
and thus
\[
\| P_k (P_{k'} \qT P_{k''} u)\|_{\dot H^{-r}} \lesssim  2^{2r (k'-k'')}    \|q\|_{\dot B_\infty^{\frac{n}{p}-2r,p}}
\| P_{k''} u\|_{\dot H^{r}} 
\]
where the $k'$ summation is trivial.

\medskip

\textbf{(ii) High-low interactions, $|k' - k| \leq 2$, $k'' \leq k+2$.}
Here we set 
\[
\tilde p =\frac{2p}{p-2}, \qquad 2<\frac{2n}{n-r}<\tilde p\leq\infty
\]
 and use 
Bernstein to place $P_{k''} u$ in $L^p$, estimating
\[
\begin{split}
\| P_k (P_{k'} \qT P_{k''} u)\|_{\dot H^{-r}} \lesssim & \  2^{-kr} \| P_{k'} \qT P_{k''} u\|_{L^2}
\\
\lesssim  & \  2^{-kr} \| P_{k'} \qT\|_{L^p} \| P_{k''} u\|_{L^{\tilde p}} 
\\ 
\lesssim & \ 2^{-kr}  2^{-(\frac{n}p -2r)k'}    \|q\|_{\dot B_\infty^{\frac{n}{p}-2r,p}(\R^n)}
2^{(\frac{n}2 - \frac{n}{\tilde p}) k''}    2^{-k'' r}  \| P_{k''} u\|_{\dot H^{r}} 
\\ 
\lesssim & \ 2^{(r-\frac{n}p)(k-k'')}    \|q\|_{\dot B_\infty^{\frac{n}{p}-2r,p}(\R^n)}  \| P_{k''} u\|_{\dot H^{r}} 
\end{split}
\]
as needed since $r < \frac{n}p$.

\medskip 

\textbf{(iii) High-high $\to $ low  interactions, $|k' - k''| \leq 2$, $k \leq k'+2$}.
Here it is more efficient to use Bernstein for the product,
\[
\begin{split}
\| P_k (P_{k'} \qT P_{k''} u)\|_{\dot H^{-r}} \lesssim & \  2^{-kr} \| P_k(P_{k'} \qT P_{k''} u)\|_{L^2}
\\
\lesssim & \  2^{( \frac{n}p  -r) k } \| P_{k'} \qT P_{k''} u\|_{L^{\tilde p'}}
\\
\lesssim  & \  2^{( \frac{n}p  -r) k }  \| P_{k'} \qT\|_{L^p} \| P_{k''} u\|_{L^2} \\
\lesssim & \ 2^{(\frac{n}p-r)(k-k'')}    \|q\|_{\dot B_\infty^{\frac{n}{p}-2r,p}(\R^n)}  \| P_{k''} u\|_{\dot H^{r}} 
\end{split}
\]
which again suffices. This concludes the proof of \eqref{dyadic} and thus the proof of the theorem.

\end{proof}

%%% ======================================================================
\section{Concentration Compactness and a d-bar Problem}
\label{Section:CC}
%%% ======================================================================

In this section we prove Theorem~\ref{ThmMain}. To recall the set-up, we seek to
show that  the d-bar operator $L_\qT$ defined by
\[
L_\qT u = \dzb u + \qT \ol u, \qquad L_q : \dot H^\frac12 \to \dot H^{-\frac12}
\]
is invertible for all $q \in L^2$, and further that its inverse satisfies the
locally uniform bound
\begin{equation}
\| L_\qT^{-1}\|_{\dot H^{-\frac12} \to \dot H^\frac12} \leq
C(\|q\|_{L^2}).
\end{equation}

We remark that the main novelty here, and the difficult part, is
the fact that we can bound the norm of $L_q^{-1}$ uniformly for
$q$ in a bounded set in $L^2$. In turn, the  key ingredient in the
proof is a non-standard use of the method of profile
decompositions, as introduced in  \cite{G98}.

We begin with several preliminaries. We first recall some basic properties of the solid Cauchy
transform $\Icb$. For a proof, see for instance \cite{Nac} [Lemma 1.4].

\begin{lemma}\label{Nac:Lemma1.4}
a) If $h\in L^{p}$, $1<p<2$ and $1/p^*= 1/p-1/2$ then
\begin{align}\label{BasicIneq}
\|\Icb h\|_{L^{p^*}}\leq c_p\|h\|_{L^{p}},
\end{align}

b) If $f\in L^{p_1}\cap L^{p_2}$, with $1<p_1<2<p_2$, then the
function $u = \Icb f$ satisfies
\begin{align*}
\|u\|_{L^\infty}\leq c_{p_1,p_2}(\|f\|_{L^{p_1}}+\|f\|_{L^{p_2}}),
\end{align*}
\begin{align*}
|u(z_1)-u(z_2)|\leq
c_{p_2}|z_1-z_2|^{1-\frac{2}{p_2}}\|f\|_{L^{p_2}}
\end{align*}
and $\lim_{|z|\rightarrow \infty}u(z)=0$.
\end{lemma}

Next we prove a qualitative result, which asserts that
$L_q^{-1}$ is a well defined operator from $L^{\frac43}$ to $L^4$:

\begin{lemma}\label{perturb}
\noindent Let $\qT\in L^2$. Then for any $f\in L^\frac43$, the equation
\begin{align}\label{pseudo}
L_{\qT}  u = f
\end{align}
has a unique solution $u \in L^4$.
\end{lemma}

\begin{proof}
 Write $\qT = \qT_n+\qT_s$, where $\qT_n\in
L^{p_1}\cap L^{p_2}$ with $1<p_1<2<p_2$ and $\|\qT_s\|_{L^2}$
small (see (\ref{rest}) below). Given a solution $u\in L^4$ of
(\ref{pseudo}), we define
\begin{align}\nu :=
\begin{cases}
e^{\Icb \big(\qT_n \frac{\ol {u}}{u}\big)}& \text{ if }u\neq 0\\
1 &\text{ if }u = 0.
\end{cases}
\end{align}
Then $\nu$ and $1/\nu$ are in $L^\infty$, in view
of Lemma \ref{Nac:Lemma1.4}. So $u\nu\in L^4$. Further, since $\nu\in W^{1,p}$ for $p\in[p_1,p_2]$ then we may apply the Leibnitz rule:
\begin{align}\label{prod}
\dzb(u\nu)=(\dzb u + \qT_n \ol {u})\nu = (- \qT_s \ol u + f)\nu.
\end{align}
Thus, using (\ref{BasicIneq}) we have
\begin{align}\label{decomp_norm} \|u
\nu\|_{L^4} \leq c\|\qT_s\|_{L^2} \|u \nu\|_{L^4} +
c\|f\|_{L^\frac43}\|\nu\|_{L^\infty}.
\end{align}
To prove uniqueness for (\ref{pseudo}), let $f=0$ and choose
$\qT_s$ with
\begin{align}\label{rest} \|\qT_s\|_{L^2}\leq
1/2c,
\end{align}
Then (\ref{decomp_norm}) yields
\begin{align*}
\|u \nu\|_{L^4} \leq \frac{1}{2} \|u \nu\|_{L^4},
\end{align*}
so $u=0$.

To show existence, we write (\ref{pseudo}) as
\begin{align}\label{integ}
\mathcal B u = \Icb f, \qquad \mathcal B = I+\Icb(\qT\, \ol\cdot).
\end{align}
The operator $\Icb( \qT\, \ol\cdot)$ is compact $L^4\rightarrow
L^4$ (see Lemma \ref{compactness}). It follows by the Fredholm
alternative that $\mathcal B$ is invertible in the $L^4\rightarrow
L^4$ topology, and we can solve for $u=\mathcal B^{-1}\Icb f$ in
$L^4$.
\end{proof}

We continue with an easy extension of the previous Lemma.
\begin{lemma}\label{LemLqInv}
  For each $q \in L^2$, the operator $L_q: \dot H^\frac12 \to \dot
  H^{-\frac12}$ is invertible and
\begin{equation}\label{Lq-inv_temp}
\|L_q^{-1}f\|_{\dot H^\frac12} \leq C(q) \|f\|_{ \dot
H^{-\frac12}}.
\end{equation}
\end{lemma}

\begin{proof}
Multiplication by $\qT$ maps $\dot H^{\half}$ to $\dot
H^{-\half}$, and we may rewrite (\ref{LqEq}) as
\begin{align}\label{integ_}
\mathcal B u = \Icb f
\end{align}
where $\mathcal B = (I+\Icb(\qT\ol\cdot)):\dot H^\half\rightarrow
\dot H^\half$. Since, $\dot H^\half\subset L^4$, injectivity
follows from Lemma \ref{perturb}.

It remains to prove surjectivity. Let $f\in \dot H^{-\half}$. Then
$\Icb f\in \dot H^\half\subset L^4$, so the proof of Lemma
\ref{perturb} yields a solution $u\in L^4$ of $\mathcal B u =
\Icb f$, i.e.
\begin{align*}
u = -\Icb(\qT\ol u) + \Icb f.
\end{align*}
Since $\qT\ol u\in L^\frac43\subset\dot H^{-\half}$, we have
$\Icb(\qT\ol u)\in \dot H^\half$ hence also $u\in \dot H^\half$
and $L_\qT u = f$.
\end{proof}

% ------------------------------------------------------------------------

By the last lemma, the best constant $C(q)$ in
\eqref{Lq-inv} is well-defined and finite for each $q \in L^2$.
The next step is to study the dependence of $L_q^{-1}$ and of
$C(q)$ on $q$:

\begin{lemma} \label{l:cont}
  The operator $L_q^{-1}$ depends smoothly on $q \in L^2$, and the
  best constant $C(q)$ in \eqref{Lq-inv} has a local Lipschitz
  dependence on $q$. More precisely, given $q_0 \in L^2$ there exists
  $\epsilon > 0$, depending only on $C(\qT_0)$, so that within the ball
  $B(q_0,\epsilon)$ the map
\[
q \to L_q^{-1}
\]
is analytic, with a uniform Lipschitz bound
\begin{equation}
\| L_{q_1}^{-1} - L_{q_2}^{-1}\|_{\dot H^{-\frac12} \to \dot
  H^{\frac12}} \lesssim C(\qT_0)^2 \|q_1 - q_2\|_{L^2}
\end{equation}
as well as
\begin{equation}
|C(q_1) - C(q_2)| \lesssim C(\qT_0)^2 \|q_1 - q_2\|_{L^2}
\end{equation}
\end{lemma}

% BPBPBPBP - LemlCont
\begin{proof}
For $q \in B(q_0,\epsilon)$ we rewrite the equation
\[
L_q u = f
\]
as
\[
L_{q_0} u = (q_0-q) \bar u + f
\]
and further as
\[
u = L_{q_0}^{-1} f + L_{q_0}^{-1} ( (q_0-q) \bar u ).
\]
If $\|q-q_0\|_{L^2} \ll C(\qT_0)^{-1}$ then the above equation can
be solved by a Neumann series. In particular we obtain the
analytic dependence of $u$ on $q$, as well as the bounds
\[
\| u\|_{\dot H^\frac12} \lesssim C(q_0) \|f\|_{\dot H^{-\frac12}}
\]
and
\[
\| u - L_{q_0}^{-1} f\|_{\dot H^\frac12}  \lesssim C(q_0)^2 \|
q_0-q\|_{L^2} \|f\|_{\dot H^{-\frac12}}.
\]
The latter leads to the desired Lipschitz bound for $L_q^{-1}$, by
repeating the same argument with $q,q_0$ replaced by $q_1, q_2$ in
the same ball.
\end{proof}
% EPEPEPEP - LemlCont

It remains to prove that the $C(q)$ bound is uniform for $q$ in a
bounded set in $L^2$. We denote by
\[
C(R) = \sup \{ C(q); \ \|q\|_{L^2} \leq R\}, \qquad C: \R^+ \to
[0, \infty],
\]
We need to prove that $C(R)$ is finite for all $R > 0$. This is the
case for $R$ small, as can be seen from the proof of the previous
lemma by taking $\qT_0=0$. We also have:

\begin{lemma}\label{LemCNondec}
The function $C(R)$ is nondecreasing and continuous.
\end{lemma}

% BPBPBPBP - LemCNondec
\begin{proof}
The monotonicity is obvious. The continuity is due to the
uniformity in the previous lemma. Precisely, if $C(R-0)= \lim_{r \nearrow R} C(r)$ is finite
then  for $\|q_0\|_{L^2} < R$, the ball size $\epsilon$ in the
previous lemma depends only on $ C(R-0)$. This yields a uniform
Lipschitz constant for $C(q)$ in $B(0,R+\epsilon)$, and the
desired continuity (indeed local Lipschitz continuity) follows.
\end{proof}
% EPEPEPEP - LemCNondec

To prove that $C(R)$ is finite for all $R$ we argue by
contradiction. Choose $R_0 > 0$ minimal so that
\[
C(R_0) = \infty.
\]
Then for $R < R_0$ we have $C(R) < \infty$, and, by the continuity
property,
\[
\lim_{R \to R_0} C(R) = \infty.
\]
Thus there exists  sequence $q_n$  so that
\[
R_0 > \| q_n\|_{L^2} \to R_0
\]
and
\[
\| L_{q_n}^{-1}\|_{\dot H^{-\frac12} \to \dot H^\frac12} \to
\infty.
\]
If we knew that $q_n$ converged (say on a subsequence) to some $q
\in L^2$   then we would have
\[
\| L_{q_n}^{-1}\|_{\dot H^{-\frac12} \to \dot H^\frac12} \to \|
L_{q}^{-1}\|_{\dot H^{-\frac12} \to \dot H^\frac12}  \neq \infty
\]
which would contradict the minimality of $R_0$.

However, there are two  obvious obstructions to compactness
arising from the symmetries of the problem, namely translation and
scaling. Any such symmetry can be described using a positive
scale factor $\lambda$ and a translation distance $y$. We
introduce the notation
\[
S(\lambda,y) q = \lambda q(\lambda (x-y)).
\]
Then
\[
C(q) = C(S(\lambda,y) q ).
\]
In view of this fact, one might try to show that we have
compactness up to symmetries, i.e. that (on a subsequence) there
exist $\lambda_n, y_n$ so that
\[
S(\lambda_n,y_n) q_n \to q \qquad \text{ in } L^2.
\]
Since the constant $C(q)$ is easily seen to be invariant with respect to symmetries,
this would again lead to contradiction.

This seems to be still too much to ask.  We will prove instead a
weaker compactness statement, which will nevertheless be
sufficient to establish the finiteness of $C(R)$. As an
intermediate step in establishing a compactness property, we first
note that, in view of Theorem \ref{ThmBilEst}, we can extend the
perturbative theory to a larger space, namely
\[
q \in  \dot B^{-\frac13,3}_\infty.
\]
The exact exponents for the Besov space are not important, just
the fact that this space has negative Sobolev regularity and the
same scaling as $L^2$, and in particular we have the Sobolev
embedding
\[
L^2 \subset \dot B^{-\frac13,3}_\infty.
\]
We note below a special case of Theorem $\ref{ThmBilEst}$:
\begin{lemma}\label{LemBilinear}
The following bilinear estimate holds:
\begin{equation}
\| q u\|_{\dot H^{-\frac12}} \lesssim \|
q\|_{\dot B^{-\frac13,3}_\infty} \|u \|_{\dot H^\frac12}.
\end{equation}
\end{lemma}

\begin{proof}
See Theorem \ref{ThmBilEst}.
\end{proof}

Using this we obtain the following extension of
Lemma~\ref{l:cont}:

\begin{lemma} \label{l:cont+}
 Given $q_0 \in L^2$ there exists $\epsilon > 0$, depending only on $C(q_0)$, so that within the ball
\[
\| q - q_0 \|_{ \dot B^{-\frac13,3}_\infty} \leq \epsilon,
\]
 the map
\[
q \to L_q^{-1}
\]
is analytic, with a uniform Lipschitz bound
\begin{equation}
\| L_{q_1}^{-1} - L_{q_2}^{-1}\|_{\dot H^{-\frac12} \to \dot
  H^{\frac12}} \lesssim C(\qT_0)^2 \|q_1 - q_2\|_{ \dot B^{-\frac13,3}_\infty  },
\end{equation}
as well as
\begin{equation}
|C(q_1) - C(q_2)| \lesssim C({\qT_0})^2 \|q_1 - q_2\|_{
\dot B^{-\frac13,3}_\infty  }.
\end{equation}
\end{lemma}

The proof is identical to the proof of Lemma~\ref{l:cont}, and is
omitted. This property shows that  it would suffice to establish
the weaker convergence property
\[
S(\lambda_n,y_n) q_n \to q \qquad \text{ in }
\dot B^{-\frac13,3}_\infty.
\]

Now we return to our compactness question. The discussion above
suggests that we should look at compactness modulo symmetries. The
last lemma tells us that we only need convergence in the weaker
$\dot B^{-\frac13,3}_\infty$ topology.  Still, for an arbitrary
sequence $q_n$ which is bounded in $L^2$ even this is too much to
hope for, as the $q_n$'s may be split into pieces which are driven
by different symmetries. The situation is very accurately
described using a profile decomposition, see \cite{G98} and also
\cite{S98}:

\begin{proposition}\label{p:BG}
  Let $q_n$ be a bounded sequence in $L^2$. Then up to the extraction
  of a subsequence, it can be decomposed in the following way:
\begin{equation}
\forall l \in \N, q_n = \sum_{k = 1}^l S(\lambda_n^k,y_n^k) q^k +
q_n^l
\end{equation}
where the functions $q^j$ are in $L^2$ for all $j \in \N$, and the
remainders $q_n^l$ are uniformly bounded in $L^2$ and satisfy
\begin{equation}\label{qnl}
\lim_{l \to \infty} \limsup_{n \to \infty} \|q_n^l\|_{ \dot
B^{-\frac13,3}_\infty} = 0,
\end{equation}
and where for any $k \in \N$, $(\lambda_n^k,y_n^k)$ is a sequence
in $\R^+ \times \R^2$ with the property that for every $j \neq k$
we have either
\begin{equation}\label{orth1}
\lim_{n \to \infty} \frac{\lambda_n^j}{\lambda_n^k} +
\frac{\lambda_n^k}{\lambda_n^j} = \infty
\end{equation}
or
\begin{equation}\label{orth2}
\lambda_n^j = \lambda_n^k, \qquad \lim_{n \to \infty}
|y_n^j-y_n^k| \lambda_n^j  = \infty .
\end{equation}
Furthermore, for each $l $ we have
\begin{equation}\label{sq-sum}
\| q_n\|_{L^2}^2  = \sum_{k=1}^l \|q^k\|_{L^2}^2  +
\|q_n^l\|_{L^2}^2 + o(1)
\end{equation}
as $n \to \infty$.
\end{proposition}

We remark here that this is the elliptic version of the profile
decomposition, as opposed to the wave equation version \cite{BG}
or the Schr\"odinger version \cite{MV98}.

We also remark that the original elliptic profile decomposition of
G\'erard~\cite{G98} is for $\dot H^s$ functions with $0<s<1$. The
transition to the statement above is straightforward, simply by
choosing $s=\frac13$ and applying a $|D|^\frac13$ operator.

\bigskip

We apply this decomposition to our sequence $q_n$, and will
distinguish two scenarios:

\begin{itemize}
\item Exactly one profile. Then up to symmetries we have
\[
q_n \to q^1 \qquad \text{ in }  \dot B^{-\frac13,3}_\infty
\]
and,  according to the prior discussion, the proof of Theorem~\ref{ThmMain}
is concluded.

\item more than one profile. Then in view of \eqref{sq-sum} we
must have
\begin{equation} \label{Lqk}
\sup_k \| q^k\|_{L^2} = R < R_0.
\end{equation}
Hence in this case we control the operator norms
$\|L_{q^k}^{-1}\|$ associated to each profile uniformly, and we
will use this to control $\| L_{q_n}^{-1}\|$.
\end{itemize}
\bigskip
To eliminate the case of multiple profiles we will use the
solutions to the $L_{q^k}$ equations to construct a solution to
$L_{q_n}$.  Precisely,  in order to complete the proof of Theorem~\ref{ThmMain}
it suffices to prove the following:

\begin{proposition}\label{PropLimC}
Let $q_n$ be a bounded sequence of $L^2$ functions with a profile
decomposition as above, so that \eqref{Lqk} holds.  Then we have
\begin{equation}\label{lqn}
\limsup_{n \to \infty} C(q_n) \lesssim C(R).
\end{equation}
\end{proposition}
% (commented out by Idan on Dec 7) This lemma concludes the proof of our main theorem.

% =================================================================================================

% BPBPBPBP - PropLimC
\begin{proof}
For $f \in \dot H^{-\frac12}$ we seek to solve
\[
L_{q_n} u = f.
\]
 By Lemma~\ref{l:cont+},
the tails $q_n^l$ play a perturbative role in this analysis.
Precisely, by choosing $l$ large enough and $n$ large enough we
can insure that
\[
\| q_n^l\|_{\dot B^{-\frac13,3}_\infty} \ll_{C(R)} 1
\]
and thus neglect them. Thus, for the rest of the proof we simply
fix $l$ and assume that
\[
q_n = \sum_{k = 1}^l S(\lambda_n^k,y_n^k) q^k.
\]
Here its components $S(\lambda_n^k,y_n^k) q^k$ are localized
around $ y_n^k$ at frequency scale $\lambda_n^k$, and are
separating as $n \to \infty$. To take advantage of this we also
split $f$ in a linear fashion as
\[
f = \sum f_n^k + f_n^{out}
\]
so that $ f_n^k$ will primarily interact only with
$S(\lambda_n^k,y_n^k) q^k$, and $f_n^{out}$ does not interact with
any of the $S(\lambda_n^k,y_n^k) q^k$. Then we seek an approximate solution of the form
\[
u_n^{app} = \sum  u_n^k + u_n^{out}, \qquad    u_n^k=
L_{S(\lambda_n^k,y_n^k) q^k}^{-1}  f_n^k, \qquad
  u_n^{out} =  L_0^{-1} f_n^{out}.
\]
Thus we have
\[
L_{q_n} u_n^{app} = \sum_{k \neq j}  S(\lambda_n^k,y_n^k) q^k 
\bar u_n^j + \sum_k   S(\lambda_n^k,y_n^k) q^k  \bar u_n^{out}   + f.
\]
 To succeed, we need to insure that
we have the following properties:

\begin{enumerate}
\item[(P1)] Almost orthogonal decomposition for $f$,
\begin{equation}
\| f_n^{out}\|_{\dot H^{-\frac12}}^2+ \sum_k \| f_n^k\|_{\dot H^{-\frac12}}^2  \lesssim \|f\|_{\dot H^{-\frac12}}^2  + o_n(1) \|f\|_{\dot H^{-\frac12}}^2.
\end{equation}
\item[(P2)] Almost orthogonal decomposition for $u_n^{app}$,
\begin{equation}
\|u_n^{app}\|_{\dot H^\frac12}^2  \lesssim  \| u_n^{out}\|_{\dot
H^{\frac12}}^2 + \sum_k \| u_n^k\|_{\dot
H^{\frac12}}^2 + o_n(1) \|f\|_{\dot H^{-\frac12}}^2.
\end{equation}
\item[(P3)] Negligible off-diagonal interactions,
\begin{equation}
\begin{split}
\| S(\lambda_n^k,y_n^k) q^k u_n^j \|_{\dot H^{-\frac12}} = & \  o_n(1)
\|f\|_{\dot H^{-\frac12}} \qquad k \neq j,
\\
\| S(\lambda_n^k,y_n^k) q^k u_n^{out} \|_{\dot H^{-\frac12}} = & \  o_n(1)
\|f\|_{\dot H^{-\frac12}}.
\end{split}
\end{equation}
\end{enumerate}
Here we remark that all implicit constants should be universal.
However, all expressions $o_n(1)$, which decay to zero as $n \to
\infty$, may have a decay rate that depends on all parameters in
our problem, namely  $q^k$, $\lambda_n^k$ and $y_n^k$ (but not on $f$). It is for this reason that the $o_n(1)$ term is not included in the first term on the right in (P1). 

We first verify that these three properties (P1), (P2) and (P3)
suffice in order to prove Proposition~\ref{PropLimC}. To see that,
we observe that in view of \eqref{Lqk}, the approximate solution $u_n^{app}= u_n^{app}(f)$ satisfies
\[
\| u^{app}_n\|_{\dot H^\frac12} \lesssim (C(R) + o_n(1))\|f\|_{\dot
H^{-\frac12}}  , \qquad \| L_{q_n} u_n^{app}-f \|_{\dot
H^{-\frac12}}  \lesssim o_n(1)\|f\|_{\dot H^{-\frac12}}.
\]
If $n$ is large enough then the error can be made arbitrarily small, therefore a
simple reiteration scheme would allow us to pass from an approximate solution to an
exact solution. Thus  Proposition~\ref{PropLimC} is proved.

\bigskip

It remains to construct a decomposition with the above properties.
In order to construct the decomposition functions $f_n$ we introduce a
family of truncation operators $T^{\mu}(\lambda,y)$, where
$(\lambda,y)$ are associated to our symmetry group and $\mu \geq
1$ is an additional dimensionless scale parameter. Precisely, we
set
\[
T^{\mu}(\lambda,y) = \chi\left(\mu^{-2} \lambda (x-y)\right)
P_{[\lambda / \mu, \lambda  \mu ]}
\]
where $\chi$ is a Schwartz function with compactly supported Fourier
transform, and so that near zero we have
\[
1 - \chi(x) = O(|x|^N).
\]
The role of the support assumption is to insure that  our
operators $T^{\mu}(\lambda,y)$ are frequency localized in the
region
\[
\lambda/ \mu \leq |\xi| \leq \lambda \mu.
\]
 With this notation,  the components $f_n^k$ of $f$ are defined by
\[
f_n^k = T^{\mu_n}(\lambda_n^k,y_n^k)  f
\]
using a  slowly increasing sequence $\mu_n \to \infty$.  Here the
meaning of slowly is taken relative to the growth rates in
\eqref{orth1}, \eqref{orth2}. The reason we let $\mu_n \to \infty$
is so that in the limit this localization captures the effect of
all of $q^k$.

In order to prove that the functions $f_n^k$ and $q_n^k$ have the
desired properties we first consider some simple properties of the
operators  $T^{\mu}(\lambda,y)$. The first set of properties
involve a single scale:

\begin{lemma}\label{l:fixedk}
The operators $T^{\mu}(\lambda,y)$ have the following properties
uniformly in $(\lambda,y) \in \R^+ \times \R^2$:

(i) They are bounded in $\dot H^s$ for $|s| \leq \frac12$,
 uniformly in  $\mu \geq 1$.

(ii) We have the decay property
\begin{equation}\label{k-away}
\| (1-T^{\mu^2}(\lambda, y)) T^{\mu}(\lambda, y)\|_{\dot H^s \to
\dot H^s} \lesssim \mu^{-N}.
\end{equation}

(iii)  For $q \in L^2$ we have the bound
\begin{equation}\label{k-q}
\lim_{\mu \to \infty} \|   S(\lambda, y) q (1-T^{\mu}(\lambda,
y))  \|_{\dot H^\frac12 \to \dot H^{-\frac12}} = 0.
\end{equation}

(iv) For $q \in L^2$ we have the commutator bound
\begin{equation}\label{k-com}
\lim_{\mu \to \infty} \|  [ L_{S(\lambda, y) q}, T^{\mu}(\lambda,
y)]  \|_{\dot H^\frac12 \to \dot H^{-\frac12}} = 0.
\end{equation}
\end{lemma}

% BPBPBPBPBPBPBPBP - Lemlfixedk

\begin{proof} [Proof of Lemma \ref{l:fixedk}]
We first note that all of the properties in the lemma are scale
and translation invariant, therefore we can simply set $\lambda =
1$ and $y = 0$ and drop them from the notation.

Next, we note that the projector part of $T^{\mu}$ selects the
frequencies $[\mu^{-1},\mu]$, whereas the multiplication part is
localized at frequency $\mu^{-2}$ which is much smaller.  This implies
that $T^\mu$ maps every dyadic frequency shell into a slight
enlargement of itself. Because of this it suffices to prove the bounds
in part (i) and (ii) for functions in a fixed dyadic shell, and the
output will also be localized in a double dyadic shell.  But on a
fixed dyadic frequency shell the $\dot H^s$ weights $|\xi|^s$ have a
fixed size. Consequently the $\dot H^s$ bounds in (i) and (ii) are all
equivalent uniformly in $\mu \gg 1$, and we can simply set $s=0$.

The $L^2$ boundedness of $T^\mu$ is trivial, uniformly in $\mu$. 
 Given the frequency localization of the multiplicative part we have
\[
(1-T^{\mu^2})T^\mu = (1 - \chi(\mu^{-4} x) P_{[\mu^{-2},\mu^2]})
\chi(\mu^{-2} x) P_{[\mu^{-1},\mu]} =  (1 - \chi(\mu^{-4} x))
\chi(\mu^{-2} x) P_{[\mu^{-1},\mu]}
\]
so (ii) also follows.

For (iii) we write
\[
q(1-T^\mu) =(1-P_{[\mu^{-\frac12},\mu^{\frac12}]}) q
(1-P_{[\mu^{-1},\mu]}) + P_{[\mu^{-\frac12},\mu^{\frac12}]} q
(1-P_{[\mu^{-1},\mu]}) + q(1-\chi(\mu^{-2} x)) P_{[\mu^{-1},\mu]}
\]
and use the multiplicative property $L^2 \cdot \dot H^\frac12 \subset \dot H^{-\frac12}$.
The first and the last term decay since
$(1-P_{[\mu^{-\frac12},\mu^{\frac12}]}) q $ and $q(1-\chi(\mu^{-2}
x))$ decay to zero in $L^2$. The middle term decays due to the
increasing frequency separation between the two factors. Precisely,
by a careful use of Bernstein's inequality we have the dyadic bound
\[
\|P_{\lambda_1} q P_{\lambda_2} u \|_{\dot H^{-\frac12}} \lesssim \min\left\{
\frac{\lambda_1}{\lambda_2},\frac{\lambda_2}{\lambda_1}
\right\}^\frac12 \|P_{\lambda_1} q \|_{L^2} \|P_{\lambda_2} u \|_{\dot H^{\frac12}}
\]
which after dyadic summation yields 
\[
\| P_{[\mu^{-\frac12},\mu^{\frac12}]} q (1-P_{[\mu^{-1},\mu]})
u\|_{\dot H^{-\frac12}} \lesssim \mu^{-\frac14} \| q\|_{L^2} \|
u\|_{\dot H^{\frac12}}
\]

For (iv) we treat separately the $\dbar$ and the $q$ part of
$L_q$. For the $q$ part we disregard the commutator structure and
write
\[
[q,T^\mu] = q( T^\mu -1) - (T^\mu -1)q
\]
where the first part decays due to (iii) and the second term is
quite similar. For the $\dbar$ part we write
\[
[\dbar,T^\mu] = \mu^{-2} (\dbar \chi)(\mu^{-2} x)
P_{[\mu^{-1},\mu]}
\]
which acts separately on each dyadic frequency. The $\dot
H^\frac12 \to \dot H^{-\frac12}$ norm of $ P_{[\mu^{-1},\mu]}$ is
$\mu$, which is more than compensated for by the $\mu^{-2}$
factor. This completes the proof of Lemma \ref{l:fixedk}
\end{proof}
% EPEPEPEPEPEPEPEP - Lemlfixedk

The next lemma is related to the scale separation
properties:

\begin{lemma} \label{l:kj}
In the setting of Proposition~\ref{p:BG}, assume that $\mu_n \to
\infty$ slowly enough. Then  we have:
\begin{equation}\label{kj-orth}
\lim_{n \to \infty} \| |D|^{s_1}   T^{\mu_n}(\lambda_n^j, y_n^j)
|D|^{s_2}  T^{\mu_n} (\lambda_n^k, y_n^k)  |D|^{s_3} \|_{L^2 \to
L^2} = 0, \qquad j \neq k, \quad s_1+s_2+s_3 = 0
\end{equation}
and the similar result with either $T^{\mu_n}$ replaced by
$(T^{\mu_n})^*$.
\end{lemma}

% BPBPBPBPBPBPBPBP - Lemlkj

\begin{proof}[Proof of Lemma \ref{l:kj}]
Denote
\[
Q^{jk} = |D|^{s_1}   T^{\mu_n}(\lambda_n^j, y_n^j) |D|^{s_2}
T^{\mu_n} (\lambda_n^k, y_n^k)  |D|^{s_3}.
\]
We consider the two scenarios in \eqref{orth1} and \eqref{orth2}.
In the first case, for large enough $n$ the operators
$T^{\mu_n}(\lambda_n^j, y_n^j)$ and $T^{\mu_n} (\lambda_n^k,
y_n^k)$ have disjoint frequency localizations so $Q^{jk} = 0$.

In the second case  we have $\lambda_n^k = \lambda_n^j:=
\lambda_n$. Further, all operators in $ Q^{jk} $ act separately on
different dyadic shells, so by orthogonality we can fix the input
frequency and insert dyadic frequency localizations in all
multipliers. Thus it suffices to consider operators of the form
\[
\tilde Q^{jk} = R_\lambda^1 \chi(\mu_n^{-2} \lambda_n (x-y_n^k))
R_\lambda^2 \chi(\mu_n^{-2} \lambda_n (x-y_n^j )) R_\lambda^3
\]
where $R_\lambda^i$ are smooth bounded multipliers localized at a
frequency $\lambda \in [\mu_n^{-1} \lambda_n, \mu_n \lambda_n]$
(which also depends on $n$). The condition $s_1+s_2 + s_3 =0$ guarantees
 that the Sobolev weights $|\xi_1|^{s_1} |\xi_2|^{s_2} |\xi_3|^{s_3}$ cancel out
if $|\xi_j| \approx \lambda$. This applies equally whether we work
with the operators $T^{\mu_n}(\lambda_n^j, y_n^j)$ or with their
adjoints. Set 
\[
\tilde \mu_n^{-1} \lambda= \mu_n^2\lambda_n, \qquad \tilde \mu_n    \in [\mu_n,\mu_n^3].
\]
We can rescale to set $\lambda = 1$, with  $y^k_n$ rescaled accordingly. Then $\tilde
Q^{jk}$ become
\[
\tilde Q^{jk} = R^1 \chi(\tilde \mu_n^{-1} (x-y_n^k)) R^2
\chi(\tilde \mu_n^{-1} (x-y_n^j)) R^3
\]
where $\tilde \mu_n    \in [\mu_n,\mu_n^3]$ goes to infinity slowly
enough,  so that also
\[
\tilde \mu_n^{-1}|  y_n^k - y_n^j| \to \infty.
\]
Then the operators $R^i$ have uniformly bounded Schwartz kernels,
so the desired conclusion follows from the spatial separation of
the two bump functions. This completes the proof of Lemma
\ref{l:kj}.
\end{proof}
% EPEPEPEPEPEPEPEP - Lemlkj

It remains to use the two lemmas above in order to prove the three
properties (P1), (P2) and (P3).

\medskip

For (P1) it is easily seen that the operators $\sum_k T^{\mu_n}(\lambda_n^k,y_n^k)$ are bounded
in $\dot H^{-\frac12}$, uniformly for large $n$, therefore it remains to show 
that in the limit $f_n^k$ are almost orthogonal,
\[
\lim_{n \to \infty} \langle f_n^k, f_{n}^j \rangle_{\dot
H^{-\frac12}}  = 0, \qquad k \neq j
\]
To see this we write
\[
\langle f_n^k, f_{n}^j \rangle_{\dot H^{-\frac12}} = \langle
|D|^{-\frac12} Q^{kj} |D|^{-\frac12} f, f \rangle
\]
where
\[
Q^{kj} = |D|^{\frac12} (T^{\mu_n}(\lambda_n^k, y_n^k))^* |D|^{-1}
T^{\mu_n}(\lambda_n^j, y_n^j)  |D|^{\frac12}
\]
Then it suffices to show that
\[
\lim_{n \to \infty} \| Q^{kj} \|_{L^2 \to L^2} = 0
\]
which follows from \eqref{kj-orth}.

\medskip

Now we consider the property (P2), for which it suffices to show
that $u_n^k$ are almost orthogonal in the limit,
\[
\lim_{n \to \infty}  \langle u_n^k, u_n^j \rangle_{\dot H^\frac12}
= 0, \qquad k \neq j.
\]
Unfortunately $u_n^k$ no longer share the sharp localization of
$f_n^k$. However, the bulk of $u_n^k$ does. Precisely, we split
\begin{equation}\label{unk-decomp}
u_n^k =  T^{\mu_n^2}(\lambda_n^k, y_n^k) u_n^k +
(1-T^{\mu_n^2}(\lambda_n^k, y_n^k)) u_n^k.
\end{equation}
For the first term the same argument as the one used above for $f_n^k$ applies, except that
we need to use the operators
\[
\tilde {\tilde Q}^{kj} = |D|^{-\frac12} T^{\mu_n^2}(\lambda_n^k, y_n^k)^* |D|
T^{\mu_n^2}(\lambda_n^j, y_n^j)  |D|^{-\frac12}
\]
and show that
\[
\lim_{n \to \infty} \| \tilde {\tilde Q}^{kj} \|_{L^2 \to L^2} = 0.
\]
This again is a consequence of \eqref{kj-orth}.

The second term, on the other hand, converges to $0$. To see that
we compute
\[
L_{S(\lambda_n^k, y_n^k)q^k}  (1-T^{\mu_n^2}(\lambda_n^k, y_n^k))
u_n^k = (1-T^{\mu_n^2}(\lambda_n^k, y_n^k)) T^{\mu_n}(\lambda_n^k,
y_n^k) f - [ L_{S(\lambda_n^k, y_n^k)q^k}, T^{\mu_n^2}(\lambda_n^k,
y_n^k)] u_n^k
\]
where we want to show that both terms decay to zero in $\dot
H^{-\frac12}$. But this follows from \eqref{k-away} for the first
term, respectively \eqref{k-com} for the second.
\medskip

Finally we consider the last property (P3). First we use the same
decomposition \eqref{unk-decomp} as above for $u_n^k$ to write
\[
\begin{split}
S(\lambda_n^j, y_n^j) q^j  u_n^k = & \ S(\lambda_n^j, y_n^j) q^j
(1-  T^{\mu_n^2}(\lambda_n^k, y_n^k)) u_n^k + S(\lambda_n^j,
y_n^j) q^j  T^{\mu_n^2}(\lambda_n^j, y_n^j)
T^{\mu_n^2}(\lambda_n^k, y_n^k) u_n^k
\\ & \ +  S(\lambda_n^j, y_n^j) q^j  (1-T^{\mu_n^2}(\lambda_n^j, y_n^j))  T^{\mu_n^2}(\lambda_n^k, y_n^k)
u_n^k.
\end{split}
\]
The first term  decays to zero since its second factor $(1-  T^{\mu_n^2}(\lambda_n^k, y_n^k)) u_n^k$
decays to zero in $\dot H^\frac12$, as established above.  For the second we need
to show that
\[
\lim_{n \to \infty} \| |D|^{\frac12}   T^{\mu_n^2}(\lambda_n^j,
y_n^j) T^{\mu_n^2}(\lambda_n^k, y_n^k)  |D|^{-\frac12} \|_{L^2 \to
L^2} = 0, \qquad j \neq k
\]
which is a consequence of \eqref{kj-orth}. For the third we need
\begin{equation}
\lim_{n \to \infty} \| |D|^{-\frac12} S(\lambda_n^j, y_n^j) q^j
(1-   T^{\mu_n}(\lambda_n^j, y_n^j))  |D|^{-\frac12} \|_{L^2 \to
L^2} = 0.
\end{equation}
which follows from \eqref{k-q}.

Finally for the outer part we write
\[
u^{out}_n = \dbar^{-1} (1- T^{\mu_n}(\lambda_n^k, y_n^k))  f -
\sum_{j \neq k}   \dbar^{-1} T^{\mu_n}(\lambda_n^j, y_n^j)  f.
\]
The summand in the second term is nothing but  $u_n^j$ evaluated in the 
special case when $q_j = 0$.  Hence, it is covered by the prior analysis. 
For the first term we write
\[
\dbar^{-1} (1- T^{\mu_n}(\lambda_n^k, y_n^k)) f = (1-
T^{\mu_n}(\lambda_n^k, y_n^k)) \dbar^{-1} f + \dbar^{-1} [L_0,
T^{\mu_n}(\lambda_n^k, y_n^k)] \dbar^{-1} f.
\]
For the first term we get decay when matched against $q^k$, by
\eqref{k-q}. For the second we disregard $q^k$ and use instead the
commutator bound \eqref{k-com}. This completes the proof of
Proposition \ref{PropLimC}.
\end{proof}

%%% ======================================================================
\section{The Scattering Transform}\label{Section:Scattering}
%%% ======================================================================

The equations (\ref{DSIIdefoc}) arise as the compatibility condition
of the Lax pair
\begin{align}\label{LaxI}
\begin{cases}
\dzb m^1 &= \qT m^2\\
(\dz+ik)m^2 &= \ol \qT m^1
\end{cases}
\end{align}
and
\begin{align}\label{LaxII}
\begin{cases}
&i\partial_t m^1 + \dz^2 m^1 + 2ik\dz m^1 -\qT\dzb m^2 + \dzb \qT
m^2 + 4 gm^1=0\\
-&i\partial_t m^2 + \dzb^2 m^2 - ik\dz m^2 -\ol\qT\dz m^1 + \dz
\ol\qT m^1 + 4 \ol gm^2=0.
\end{cases}
\end{align}
The construction of the Scattering Transform only involves
solutions of the Dirac system (\ref{LaxI}). The equations
(\ref{LaxII}) are used afterwards, in order  to establish its time evolution
(\ref{tT_time}).

Assuming $\qT$ is a Schwartz function,  Beals and Coifman \cite{BC} studied Jost-type
solutions to (\ref{LaxI}) with boundary conditions
\begin{align}\label{bdry}
\begin{cases}
m^1\rightarrow 1 \text{ as }|z|\rightarrow \infty\\
m^2\rightarrow 0 \text{ as }|z|\rightarrow \infty.
\end{cases}
\end{align}
With the substitutions
\begin{align}\label{m1m2_def}
m_\pm = m^1\pm e_{-k}\ol{m^2}
\end{align}
they obtained the decoupled pseudo-analytic equations (\ref{eq_n})
which we introduced in Section \ref{Section:Intro}. They also
established the dual set of equations
\begin{align}\label{k_deriv_m}
\begin{cases}
\displaystyle \dkb m^1 = e_{-k}\tT\ol{m^2}\\
\displaystyle \dkb m^2 = e_{-k}\tT\ol{m^1}
\end{cases}
\end{align}
which is equivalent to (\ref{eq_m}).

Throughout this section  we will use both  $m^1$ and $m^2$ as well
as  the functions $m_\pm$ defined in Section \ref{Section:Intro}. We have
\begin{align}\label{m1m2}
\begin{split}
m^1 &= \half(m_++m_-) = \half (n_++n_-)\\
m^2 &= \half e_{-k}(\ol{m_+-m_-}) = \half(n_+-n_-).
\end{split}
\end{align}
For the Scattering Transform (\ref{tT}) we will also use the
expression
\begin{align}\label{tT_}
\mathcal S\qT(k) = -\frac{i}{\pi}\int_{\R^2} e_k(z) \ol{q(z)}
m^1(z,k) dz.
\end{align}

Our goal is to solve \eqref{eq_n} for $\qT\in L^2$ and show that
the corresponding scattering data $\tT$ is in $L^2$. To get
started we rewrite the equations \eqref{eq_n} in terms of the
functions $m_{\pm} -1$, which have the virtue that they decay at
infinity:
\[
\frac{\partial}{\partial \bar z} (m_{\pm} -1) =  \pm e_{-k} \qT
(\ol {m_{\pm}} -1) \pm   e_{-k} \qT.
\]
The $L^4$ solvability for equations of this type is considered in
the next lemma.

\begin{lemma}\label{perturb2}
Suppose $\qT\in L^2$. Then for any $f\in L^2$ and any $k\in \C$ such that $M\hat
f(k)<\infty$, there is a unique solution $u(\cdot,k)\in L^4$ of
\begin{align}\label{pseudo2}
\dzb u + e_{-k}\qT \ol u = e_{-k}f.
\end{align}
Moreover
\begin{align}\label{bound_u}
\|u(\cdot,k)\|_{L^4}\leq    C(\|\qT\|_{L^2}) \| \Icb(e_k\ol f)\|_{L^4} \leq
C(\|\qT\|_{L^2})\|f\|^\half_{L^2}\big(M\hat f(k)\big)^\half.
\end{align}
\end{lemma}

\begin{proof}
 The uniqueness follows from Lemma~\ref{perturb}.  To prove existence, we
  first recall that in view of Corollary \ref{FI_}b),
  $\Icb(e_{-k}f)\in L^4$ for a.e. $k$. Write
\begin{align}\label{u_def}
u = v+\Icb(e_{-k}f).
\end{align}
Then $u$ is a solution of (\ref{pseudo2}) if and only if $v$
solves
\begin{align}\label{eq_v}
\dzb v + e_{-k}\qT\, \ol v = -e_{-k}\qT\Ic(e_k\ol f).
\end{align}
The term on the right is in $L^\frac43$ for a.e. $k$. More
precisely, by Corollary~\ref{FI_} we have
\begin{align*}
\|e_{-k}\qT\Ic(e_k\ol f)\|_{L^\frac43}\leq
c\|\qT\|_{L^2}\|f\|^\half_{L^2}\big(M\hat f(k)\big)^\half.
\end{align*}
Thus, by Theorem~\ref{ThmMain} there is a unique solution $v\in \dot
H^{\frac12} \subset L^4$ for (\ref{eq_v}). Hence the function  $u$
defined by (\ref{u_def}) solves (\ref{pseudo2}) and satisfies
(\ref{bound_u}).
\end{proof}

We are now ready to construct the Jost solutions $m_\pm$ for (\ref{eq_n}):

\begin{lemma}(Jost Solutions)\label{l:Jost}
Suppose that $\qT\in L^2$, then:
\medskip

\noindent a) For almost every $k$ there exist unique solutions $m_\pm(z,k)$ of (\ref{eq_n}) with
$m_\pm(\cdot,k)-1\in L^4$ and moreover,
\begin{align}\label{L4_n}
\|m_\pm(\cdot,k)-1\|_{L^4}  + \|m^1(\cdot,k)-1\|_{L^4}  + \|m^2(\cdot,k)\|_{L^4}  &
\leq C(\|\qT\|_{L^2})\big(M\hat\qT(k)\big)^\half.
\end{align}
In addition we have
\begin{align}\label{L4_zk_n}
\|m_\pm-1\|_{L^4_kL^4_z} +  \|m^1-1\|_{L^4_kL^4_z}
+\|m^2\|_{L^4_kL^4_z}  \leq C(\|\qT\|_{L^2}),
\end{align}
as well as
\begin{align}\label{L43L4}
\| \dbar m^1(\cdot,k) \|_{L^{\frac43}} \leq
C(\|\qT\|_{L^2})\big(M\hat\qT(k)\big)^\half.
\end{align}
\noindent b) The maps $q \to m_{\pm}$, $q \to m^{1}$ and $q \to
m^{2}$ are locally Lipschitz from $L^2$ into the topologies in
\eqref{L4_zk_n}, \eqref{L43L4}. Precisely, given $\qT_1$ and
$\qT_2$ in $L^2$ we have the difference bounds
\begin{align}\label{L4_zk_n-diff}
\|\delta m_\pm\|_{L^4_kL^4_z} +  \|\delta m^1\|_{L^4_kL^4_z}
+\|\delta m^2\|_{L^4_kL^4_z}  \leq
C(\|\qT_1\|_{L^2})C(\|\qT_2\|_{L^2}) \| \delta \qT\|_{L^2}
\end{align}
as well as
\begin{align}\label{L43L4-diff}
\| \dbar \delta m^1 \|_{L^4_k L^{\frac43}_z} \leq
C(\|\qT_1\|_{L^2})C(\|\qT_2\|_{L^2}) \| \delta \qT\|_{L^2}.
\end{align}
\end{lemma}

\begin{remark}
If we use the first part of \eqref{bound_u} in the proof below  then we obtain the more refined bound
\begin{equation}\label{better-m}
 \|m^1-1\|_{L^4_kL^4_z} +\|m^2\|_{L^4_kL^4_z} + \| \dbar m^1 \|_{L^4_k L^{\frac43}_z} \leq C(\|\qT\|_{L^2})
 \|  \Icb(e_k\ol \qT)\|_{L^4_kL^4_z}.
\end{equation}
\end{remark}

\begin{proof}
a) We define
\begin{align}\label{eq_r}
r_\pm(\cdot,k)=m_\pm(\cdot,k)-1.
\end{align}
Then $m_\pm$ solve (\ref{eq_n}) if and only if $r_\pm$ solve
\begin{align}\label{diff_eq_r}
\dzb r_\pm = \pm e_{-k}\qT\ol {r_\pm} \pm e_{-k}\qT
\end{align}
and so by Lemma \ref{perturb2} there exist unique solutions to
(\ref{diff_eq_r}) with
\begin{align}\label{r_est}
\|r_\pm(\cdot,k)\|_{L^4}&\leq C(\|\qT\|_{L^2})\big(M\hat\qT(k)\big)^\half.
\end{align}
Now the bound (\ref{L4_n}) for $m_\pm$, $m^{1}$ and $m^{2}$
follows from (\ref{r_est}) and (\ref{eq_r}).

The inequality (\ref{L4_zk_n}) then follows by integrating
(\ref{L4_n}) in $k$ and using the mapping property $M:
L^2\rightarrow L^2$. Finally, for \eqref{L43L4} we use the first
equation in \eqref{LaxI} combined with the $m^2$ bound in
\eqref{L4_zk_n}. This completes the proof of a).

Part (b) is easily obtained by repeating the same arguments in
part a) for differences of Jost functions. The details are left
for the reader.
\end{proof}

Next we turn our attention to the scattering transform $\tT$ of $\qT$, which is defined by
\eqref{tT_}.

\begin{lemma}\label{l:ST}
The scattering transform $\tT(k)$ is well defined for a.e. $k$ in
$\C$ and satisfies
\begin{align}\label{tT-L2}
\|\tT\|_{L^2}\leq C(\|\qT\|_{L^2})
\end{align}
as well as the pointwise bound
\begin{align}\label{tT-point}
|\tT(k)|\leq C(\|\qT\|_{L^2}) M\hat\qT(k).
\end{align}
\end{lemma}

\begin{remark}
Using the slightly stronger bound \eqref{better-m} in the proof
below yields the following slight improvement over
\eqref{tT-point}:
\begin{align}\label{tT-point++}
\| \tT - \hat\qT(k)\|_{L^2} \leq C(\|\qT\|_{L^2})  \|  \Icb(e_k\ol
\qT)\|_{L^4}
\end{align}
This will be useful later on in order to provide a self-contained proof
of the characterization of the wave operators for the DSII problem.
\end{remark}
\begin{proof}
We write $\tT(k)$ in the form
\begin{align}\label{eq_tT}
i\tT(k)=\frac{1}{\pi}\int e_k\ol \qT \, dz + \frac{1}{\pi}\int e_k\ol
\qT (m^1-1)\, dz.
\end{align}
The first term is simply the Fourier transform of $\ol\qT\in L^2$
which obeys \eqref{tT-L2} and \eqref{tT-point}. For the second
term, we apply Theorem \ref{ThmPsiDO} with $n=2$ and $\alpha =1$
for the symbol $m^1(z,k)-1$, and $f=\ol{\hat q}$ (so that $\hat
f=\ol q$), $k$ playing the role of $x$ and $z$ playing the role of
$\xi$. Hypothesis i) of Theorem \ref{ThmPsiDO} is satisfied by
(\ref{L4_zk_n}). To see that hypothesis ii) is justified, recall from \eqref{L43L4} that
\begin{align*}
\| \dbar m^1(\cdot,k) \|_{L^{\frac43}} \leq
C(\|\qT\|_{L^2})\big(M\hat\qT(k)\big)^\half.
\end{align*}
Hence,
\begin{align*}
\| \dbar (m^1-1) \|_{L^4_kL^{\frac43}_z} \leq
C(\|\qT\|_{L^2})\|M\hat\qT(k)\|_{L^2}^\half \leq C(\|\qT\|_{L^2}).
\end{align*}
Thus, hypothesis ii) holds by the boundedness of the Beurling transform $\dzb\Ic$ on $L^p$ for $1<p<\infty$.

It follows that $\tT$ is well defined  and is in $L^2$. In
addition, from (\ref{PointwiseBound}),
\begin{align}\label{FS_bound}
|\tT(k)|&\leq C(\|\qT\|_{L^2})(M\hat \qT(k))^\half\| \dbar m^1(\cdot,k)\|_{L^\frac43}\|\qT\|^\half_{L^2} \leq C(\|\qT\|_{L^2})M\hat
\qT(k).
\end{align}
\end{proof}

So far we have constructed the Scattering Transform $\mathcal S
\qT$ for a fixed $\qT \in L^2$. Our next goal is to establish that
$\mathcal S $ is a locally Lipschitz map. One can already view
this as a consequence of the locally Lipschitz property for the
Jost functions in Lemma~\ref{l:Jost}, but the next lemma provides elegant difference formulas from which we will also obtain additional  properties of the Scattering Transform and its derivative. These formulas (more precisely their consequences stated as Lemma \ref{LemDiffeom} a) and c)) are generalizations of facts proved in \cite{BC} on tangent maps for potentials in Schwartz space and extended in \cite{Per} and \cite{SunII}.

We'll denote by $\langle\cdot,\cdot\rangle$ the usual inner product
on $L^2$:
\begin{align*}
\langle f,g\rangle = \int \ol f g.
\end{align*}

\begin{lemma}\label{lem_iso} (Difference Formulas)
a) Given any two potentials $\qT_1$ and $\qT_2$ in $ L^2(\R^2)$ with scattering transforms $\tT_1$, respectively  $\tT_2$, we have:
\begin{align}\label{tT_Delta}
\tT_1-\tT_2 = T_{\qT_1,\qT_2}( \qT_1 - \qT_2)
\end{align}
where the linear operator $T_{\qT_1,\qT_2}$ is given by
\begin{align}\label{T_def}
T_{\qT_1,\qT_2} f (k)  = -\frac{i}{\pi}\Big(  \int  e_k(z)
\ol{f(z)}a(z,k)\dnuz -  \int
e_k(z){f(z)}b(z,k)\dnuz\Big)
\end{align}
with
\begin{align}
a(z,k) &= \ol{m^1_{\ol{\qT_2}}(z,-k)}m^1_{\qT_1}(z,k)\\
b(z,k) &= \ol{m^2_{\ol{\qT_2}}(z,-k)}m^2_{\qT_1}(z,k).
\end{align}
The integrals are well defined for $f \in L^2$ and
\begin{align}\label{L2_diff}
\| T_{\qT_1,\qT_2}\|_{L^2 \to L^2}  \leq
C(\|\qT_1\|_{L^2})C(\|\qT_2\|_{L^2}).
\end{align}
b) With the same functions $a(z,k)$ and $b(z,k)$ (defined in terms of $\qT_1$ and $\qT_2$) as above we also have
\begin{align}\label{qT_Delta}
\qT_1-\qT_2 = W_{\qT_1,\qT_2}( \tT_1 - \tT_2)
\end{align}
where the linear operator $W_{\qT_1,\qT_2}$ is defined as:
\begin{align}\label{W_def}
W_{\qT_1,\qT_2} g (z)  = -\frac{i}{\pi}\Big(  \int  e_k(z)
\ol{g(k)}a(z,k)\dnuk -  \int
e_{-k}(z){g(k)}\ol{b(z,k)}\dnuk\Big).
\end{align}
The integrals are well defined for $f \in L^2$ and
\begin{align}\label{L2_diff_k}
\| W_{\qT_1,\qT_2}\|_{L^2 \to L^2}  \leq
C(\|\qT_1\|_{L^2})C(\|\qT_2\|_{L^2}).
\end{align}
c) For the operators $T_{\qT_1,\qT_2}$ and $W_{\qT_1,\qT_2}$ we have the identity:
\begin{align}\label{pre_quasi_isom}
\langle g,T_{\qT_1,\qT_2}f\rangle - \langle f,W_{\qT_1,\qT_2}g\rangle = \Im\langle g,\tilde T_{\qT_1,\qT_2}f\rangle,
\end{align}
valid for any $f,g,\qT_1,\qT_2\in L^2$, where
\begin{align}\label{T_tilde}
\tilde T_{\qT_1,\qT_2}f(k) = -\frac{2}{\pi}\int e_k(z)f(z) b(z,k)dz.
\end{align}
\end{lemma}
\bigskip
\begin{proof}
a) First we will prove (\ref{tT_Delta}) formally. We will then show
that the integral exists in $L^2$ and prove the bound 
(\ref{L2_diff}).

From the definition (\ref{tT_}), we have
\begin{align*}
\tT_1 - \tT_2 &= -\frac{i}{\pi}\Big(\int
e_k(\ol{\qT_1}-\ol{\qT_2})m^1_{\qT_1}dz + \int
e_k\ol{\qT_2}(m^1_{\qT_1}-m^1_{\qT_2})dz\Big)
\end{align*}
where $m^1_{q_i}$ solve $(\ref{LaxI})$ with boundary conditions
(\ref{bdry}), or in integral form
\begin{align*}
m^1_{\qT_i}(\cdot,k) - 1 = (I-\mathcal A_{\qT_i,k})^{-1}\mathcal
A_{\qT_i,k}(1)
\end{align*}
where
\begin{align*}
\mathcal A_{\qT,k}(\cdot) = \Icb(e_{-k}\qT\Ic(e_{k}\ol\qT\, \cdot)).
\end{align*}
For the second term, we have by the resolvent identity
\begin{align*}
m^1_{\qT_1}-m^1_{\qT_2} &= [(I-\mathcal A_{\qT_1,k})^{-1} -
(I-\mathcal A_{\qT_2,k})^{-1}]\mathcal A_{\qT_1,k}1 + (I-\mathcal
A_{\qT_2,k})^{-1}(\mathcal A_{\qT_1,k}1 - \mathcal A_{\qT_2,k}1)\\
&= (I-\mathcal A_{\qT_2,k})^{-1}\{[(I-\mathcal A_{\qT_2,k}) -
(I-\mathcal A_{\qT_1,k})](I-\mathcal A_{\qT_1,k})^{-1}\mathcal
A_{\qT_1,k}1 + \mathcal A_{\qT_1,k}1 - \mathcal A_{\qT_2,k}1\}\\
&= (I-\mathcal A_{\qT_2,k})^{-1}(\mathcal A_{\qT_1,k} - \mathcal
A_{\qT_2,k})m^1_{\qT_1}\\
&= (I-\mathcal A_{\qT_2,k})^{-1}(\mathcal D_1 + \mathcal D_2)
\end{align*}
where
\begin{align*}
\mathcal D_1 &=
\Icb(e_{-k}\qT_2\Ic(e_k(\ol{\qT_1}-\ol{\qT_2})m^1_{\qT_1}))\\
\mathcal D_2 &=
\Icb(e_{-k}(\qT_1-\qT_2)\Ic(e_k\ol{\qT_1}m^1_{\qT_1})).
\end{align*}
Then
\begin{align*}
\int e_k\ol{\qT_2}(m^1_{\qT_1}-m^1_{\qT_2})dz  &= \langle
(I-\mathcal A^*_{\qT_2,k})^{-1}e_{-k}\qT_2,\mathcal D_1 + \mathcal
D_2\rangle =\langle e_{-k}\qT_2m^1_{\ol{\qT_2}}(\cdot,-k),\mathcal
D_1 + \mathcal D_2 \rangle.
\end{align*}
Now,
\begin{align*} \langle
e_{-k}\qT_2m^1_{\ol{\qT_2}}(\cdot,-k),\mathcal D_1\rangle
&=\langle
\mathcal A_{\ol{\qT_2},-k}m^1_{\ol{\qT_2}}(\cdot,-k),e_k(\ol{\qT_1}-\ol{\qT_2})m^1_{\qT_1}
\rangle\\
&= \int
e_k(z)\ol{m^1_{\ol{\qT_2}}(z,-k)}\,\ol{(\qT_1(z)-\qT_2(z))}m^1_{\qT_1}(z,k)\dnuz\\
& - \int e_k(z)\ol{(\qT_1(z)-\qT_2(z))}m^1_{\qT_1}(z,k)\dnuz.
\end{align*}
In addition,
\begin{align*}
\langle e_{-k}\qT_2m^1_{\ol{\qT_2}}(\cdot,-k),\mathcal D_2\rangle
&= - \langle \Ic (e_{-k}\qT_2m^1_{\ol{\qT_2}}(\cdot,-k)),e_{-k}(\qT_1-\qT_2)\Ic(e_k\ol{\qT_1}m^1_{\qT_1})\rangle\\
&=  -\int e_k(z) \ol {m^2_{\ol{\qT_2}}(z,-k)} (\qT_1(z) -
\qT_2(z)) {m^2_{\qT_1}(z,k)}\dnuz.
\end{align*}
Combining the terms, we obtain (\ref{tT_Delta}). To prove
(\ref{L2_diff}), we write
\begin{align*}
T_{\qT_1,\qT_2} f  = -\frac{i}{\pi}\Big(P_1{f}(k) +
P_2{f}(k) + P_3{f}(k) + P_4{f}(k)
+ P_5{f}(k)\Big)
\end{align*}
where
\begin{align*}
P_1{f}(k) &= \int e_z\ol{f}(\ol{
m^1_{\ol{\qT_2}}}-1)({m^1_{\qT_1}}-1)dz\\
P_2{f}(k) &=\int e_z\ol{f}(\ol{ m^1_{\ol{\qT_2}}}-1)dz\\
P_3{f}(k) &=\int e_z\ol{f}({m^1_{\qT_1}}-1)dz\\
P_4{f}(k) &=\int e_z\ol{f}dz\\
P_5{f}(k) &=-\int e_{-z}
{f}\ol{m^2_{\ol{\qT_2}}}{m^2_{\qT_1}}dz.
\end{align*}
For the term $P_5$, we have by Lemma~\ref{l:Jost}
\begin{align}\label{sigma+}\nonumber\|m^2_{\ol{\qT_2}}(\cdot,-k)m^2_{{\qT_1}}(\cdot,k)\|_{L^2}
&\leq \|m^2_{\ol{\qT_2}}(\cdot,-k)\|_{L^4}
\|m^2_{{\qT_1}}(\cdot,k)\|_{L^4}\\&\leq
C(\|\qT_2\|_{L^2})\big(M\widehat{\ol{\qT_2}}(-k)\big)^\half
C(\|\qT_1\|_{L^2})\big(M\widehat{{\qT_1}}(k)\big)^\half
\end{align}
Hence,
\begin{align*}
\|P_5{f}\|_{L^2} &\leq
\Big(\int\|m^2_{\ol{\qT_2}}(\cdot,-k)m^2_{{\qT_1}}(\cdot,k)\|^2_{L^2}dk\Big)^\half\|{f}\|_{L^2}\\
&\leq C(\|\qT_1\|_{L^2})C(\|\qT_2\|_{L^2})
\|M\widehat{\ol{\qT_2}}(-\,\cdot)\|_{L^2}^\half\|M\widehat{{\qT_1}}(\cdot)\|_{L^2}
^\half\|{f}\|_{L^2}\\
&\leq C(\|\qT_1\|_{L^2})C(\|\qT_2\|_{L^2})\|{f}\|_{L^2}
\end{align*}
where the last inequality follows from the mapping property of the
maximal function $M: L^2\rightarrow L^2$. In similar fashion, we
obtain
\begin{align*}
\|P_1{f}\|_{L^2} &\leq
C(\|\qT_1\|_{L^2})C(\|\qT_2\|_{L^2})\|{f}\|_{L^2}.
\end{align*}
To investigate the term $P_3{f}$, we will  apply Theorem
\ref{ThmPsiDO} with $n=2$ and $\alpha =1$ for the symbol
$m^1_{\qT_1}(z,k)-1$, and $\ol{f}$ replacing $\hat{f}$ in the Theorem, $k$ playing the role of $x$ and
$z$ playing the role of $\xi$. The hypotheses of the theorem are
satisfied in view of the bounds \eqref{L4_zk_n} and \eqref{L43L4}
in Lemma~\ref{l:Jost}.  It thus follows from (\ref{PointwiseBound}), that
\begin{align*}
|P_3{f}(k)|&\leq C(\|\qT_1\|_{L^2}) (M\widehat f(k))^\half\|\dzb (m^1_{\qT_1}-1)\|_{L^\frac43_z}\|{f}\|^\half_{L^2}\\
&\leq C(\|\qT_1\|_{L^2}) (M\widehat {f}(k))^\half(M\hat
\qT_1(k))^\half\|{f}\|^\half_{L^2}.
\end{align*}
Hence
\begin{align*}
\|P_3{f}\|_{L^2} &\leq C(\|\qT_1\|_{L^2})
\Big(\int\big|M\widehat{f}(k)M\widehat{\qT_1}(k)\big|dk\Big)^\frac12\|{f}\|_{L^2}^\half\\
&\leq C(\|\qT_1\|_{L^2})\|M\widehat{f}\|^\half_{L^2}\|M\widehat{\qT_1}\|^\half_{L^2}\|{f}\|_{L^2}^\half\\
&\leq C(\|\qT_1\|_{L^2})\|{f}\|_{L^2}.
\end{align*}
Likewise, for $P_2$ we have
\begin{align*}
\|P_2{f}\|_{L^2} \leq C(\|\qT_2\|_{L^2})\|{f}\|_{L^2}.
\end{align*}
Combining the four terms, we obtain (\ref{L2_diff}).

b) Because of the symmetry between the forward and inverse scattering, (\ref{tT_Delta}) also yields:
\begin{align}\label{diff_q1_q2}
\qT_1(z) - \qT_2(z) = -\frac{i}{\pi}\Big(&\int
e_k(z)\ol{m^1_{\ol{\tT_2}}(-z,k)}\,\ol{(\tT_1(k) -
	\tT_2(k))}m^1_{\tT_1}(z,k)\dnuk
- \\
\nonumber &\int e_k(z)\ol{m^2_{\ol{\tT_2}}(-z,k)}{(\tT_1(k) -
	\tT_2(k))}m^2_{\tT_1}(z,k)\dnuk\Big).
\end{align}
where, by (\ref{m1m2})
\begin{align}
\begin{split}
m^1_{\tT} &= \half(n_{+,\tT}+n_{-,\tT})=m^1_{\qT}\\
m^2_{\tT} &=\half e_{-k}(\ol{n_{+,\tT}-n_{-,\tT}}) = e_{-k}\ol {m^2_{\qT}}
\end{split}
\end{align}

Now, consider the scattering problem for  $\widetilde{\qT_2} =: \ol{\qT_2}$.  For the corresponding scattering transform we have (cf. \cite{BC}): 
\begin{align*}
\widetilde{\tT_2}(k) = -\ol{\tT_2(-k)}
\end{align*}
Hence,
\begin{align*}
\dkb{n_{\pm,{\widetilde{\tT_2}}}}(z,-k) &= \pm e_{k}(z)\ol{\tT_2(k)}\,\ol{n_{\pm,{\widetilde{\tT_2}}}(z,-k)}.
\end{align*}
Comparing with the corresponding equations for $\ol{\tT_2}$ we conclude from the
uniqueness in Lemma \ref{l:Jost} that
\begin{align*}
n_{\pm,{\ol{\tT_2}}}(-z,k) = n_{\pm,{\widetilde{\tT_2}}}(z,-k).
\end{align*}
It follows (using (\ref{m1m2})) that
\begin{align*}
&{m^1_{\ol{\tT_2}}(-z,k)}={m^1_{\widetilde{\tT_2}}(z,-k)} = {m^1_{\ol{\qT_2}}(z,-k)}\\
&\ol{m^2_{\ol{\tT_2}}(-z,k)} = \ol{m^2_{\widetilde{\tT_2}}(z,-k)} = e_{-k}(z){m^2_{\ol{\qT_2}}(z,-k)}\\
\end{align*}
Substituting in (\ref{diff_q1_q2}) we obtain: 
\begin{align*}
\qT_1(z) - \qT_2(z) = -\frac{i}{\pi}\Big(&\int
e_k(z)\ol{m^1_{\ol{\qT_2}}(z,-k)}\,\ol{(\tT_1(k) -
	\tT_2(k))}m^1_{\qT_1}(z,k)\dnuk
-\\
&\int e_{-k}(z){m^2_{\ol{\qT_2}}(z,-k)}{(\tT_1(k) -
	\tT_2(k))}\ol{m^2_{\qT_1}(z,k)}\dnuk\Big)
\end{align*}
which proves (\ref{qT_Delta}). The bound (\ref{L2_diff_k}) is obtained in similar fashion to (\ref{L2_diff}). 
\goodbreak
c) Since $T_{\qT_1,\qT_2}$, $W_{\qT_1,\qT_2}$ and $\tilde T_{\qT_1,\qT_2}$ are bounded oerators on $L^2$, it suffices to prove (\ref{pre_quasi_isom}) for $f,g$ of compact support. In view of (\ref{T_def} and (\ref{W_def}) we have
\begin{align*}
\langle g,T_{\qT_1,\qT_2}f\rangle - \langle f,W_{\qT_1,\qT_2}g\rangle  
=& -\frac{i}{\pi}\int\int\ol{g(k)}\Big(e_k(z)\ol{f(z)}a(z,k)-e_k(z)f(z)b(z,k)\Big) dzdk\\
& +\frac{i}{\pi}\int\int\ol{f(k)}\Big(e_k(z)\ol{g(z)}a(z,k)-e_{-k}(z)g(z)\ol{b(z,k)}\Big) dkdz\\
= & -\frac{i}{2}\Big(\langle g,\tilde T_{\qT_1,\qT_2}f\rangle - \ol{\langle g,\tilde T_{\qT_1,\qT_2}f\rangle} \Big),
\end{align*}
which proves (\ref{pre_quasi_isom}).
\end{proof}

So far, we have established that $S$ is a Lipschitz map from $L^2(\R^2)$ to $L^2(R^2)$. The next step is to
show that the properties (1) and (5) in Theorem \ref{ThmScatteringTransform} also hold. In part a) of the following Corollary, we first use (\ref{pre_quasi_isom}) to obtain a generalization of the Plancherel identity (\ref{t1-l2}):

\begin{corollary}\label{cor_Plan}
a) For any $\qT_1,\qT_2\in L^2(\R^2)$ we have 
\begin{align}\label{quasi_isom}
\|\mathcal S\qT_1 - \mathcal S\qT_1\|^2_{L^2} = \|\qT_1 - \qT_2\|^2_{L^2} + \Im\langle\mathcal S\qT_1-\mathcal S\qT_2,\tilde T_{\qT_1,\qT_2}(\qT_1-\qT_2)\rangle,
\end{align}
with $\tilde T_{\qT_1,\qT_2}$ as defined in (\ref{T_tilde}.)
\goodbreak
b) The Scattering transform $\mathcal S$
satisfies the Plancherel identity $\| \mathcal S \qT\|_{L^2} =
\|\qT\|_{L^2}$ for all $\qT\in L^2(\R^2)$ , as well as the identity $\mathcal S^2 = I$.
\end{corollary}

\begin{proof}
  a) Apply (\ref{pre_quasi_isom}) with $f=\qT_1-\qT_2$ and $g = \mathcal S\qT_1-\mathcal S\qT_2$ and use (\ref{tT_Delta} and (\ref{qT_Delta}).
  \goodbreak
 b) The Plancherel Identity is a special case of (\ref{quasi_isom}) when $\qT_2=0$, as then $b(z,k)=0$ so $\tilde T_{\qT_1,0}=0$. It also follows by an extension by continuity argument from the Plancherel identity for potentials in Schwartz class in \cite{BC}, together with the (locally) uniformly Lipschitz continuity of $\mathcal S$ which is based only on part a) of Lemma \ref{lem_iso}. Likewise, the identity $\mathcal S^2=I$ was show in \cite{BC} for Schwartz potentials and now extends to all of $L^2$ in view of part a) of Lemma \ref{lem_iso}.
\end{proof}

We have shown that $\mathcal S$ is Lipschitz and also that $\mathcal S^{-1}=\mathcal S$, therefore $\mathcal S$ is a bi-Lipschitz homeomorphism of $L^2$. 

In order to complete the
proof of Theorem~\ref{ThmScatteringTransform}  it remains to show
that $\mathcal S $ is continuously differentiable and is a symplectomorphism.

\begin{lemma}\label{LemDiffeom}
a) The map $\qT\rightarrow \mathcal S(\qT)$ is a $C^1$-diffeomorphism from $L^2(\R^2)$ into $L^2(\R^2)$,
and its differential is given by
\begin{align*}
\Big(\frac{\delta\mathcal S}{\delta \qT}\Big|_{\qT}\tilde\qT\Big)(k) := T_{\qT,\qT} \tilde q
\end{align*}
with $T_{\qT,\qT}$ as defined in (\ref{T_def}).
\goodbreak
b)
\begin{align*}
\Big(\frac{\delta \mathcal S}{\delta \qT}\Big|_\qT\Big)^{-1} =  T_{\qT,\qT}^{-1} = W_{\qT,\qT},
\end{align*}
with $W_{\qT,\qT}$ as defined in (\ref{W_def}).
\goodbreak
c) For any $\qT,\qT_1,\qT_2\in L^2(\R^2)$ we have
\begin{align}\label{symplecto}
\Im\langle \frac{\delta \mathcal S}{\delta \qT}\Big|_{\qT}\qT_1,\frac{\delta \mathcal S}{\delta \qT}\Big|_{\qT}\qT_2\rangle = \Im \ol{\langle\qT_1,\qT_2\rangle}
\end{align}
\end{lemma}
\bigskip
\begin{proof}
a)
Given two potentials $\qT_1$ and $\qT_2$,  we need to estimate the difference
\[
\delta^{(2)} \tT = \tT_2 - \tT_1 - \Big(\frac{\delta\mathcal S}{\delta \qT}\Big|_{\qT_1}\delta \qT\Big)
\]
in terms of $ \delta \qT = \qT_2 - \qT_1$.
It suffices to show that
\begin{equation}\label{D2-est}
\| \delta^{(2)} \tT \|_{L^2} \lesssim
C(\|\qT_1\|_{L^2})C(\|\qT_2\|_{L^2}) \| \delta \qT\|_{L^2}^2.
\end{equation}
Using the formula \eqref{tT_Delta} we write
\begin{align}\label{var}
\delta^{(2)} \tT  = T_{\qT_1,\delta \qT} (\qT_2 - \qT_1)
\end{align}
where
\[
T_{\qT_1,\delta \qT} f
 = \frac{i}{\pi}\Big(\int
e_k(z)\ol{f}\delta a(z,k)\dnuz - \int
e_k(z){f}\delta b(z,k)\dnuz\Big)
\]
with
\begin{align*}
\delta a(z,k) &= \delta\ol{m^1_{\ol{\qT}}(z,-k)}m^1_{\qT_1}(z,k) + \ol{m^1_{\ol{\qT_1}}(z,-k)}\delta m^1_{\qT}(z,k)\\
\delta b(z,k) &= \delta\ol{m^2_{\ol{\qT}}(z,-k)}m^2_{\qT_1}(z,k) +
\ol{m^2_{\ol{\qT_1}}(z,-k)}\delta m^2_{\qT}(z,k).
\end{align*}
Now the bound \eqref{D2-est} is a consequence of
\[
\|  T_{\qT_1,\delta \qT} \|_{L^2 \to L^2} \leq
C(\|\qT_1\|_{L^2})C(\|\qT_2\|_{L^2}) \|\delta \qT\|_{L^2}.
\]
This in turn is proved in the same manner as \eqref{L2_diff}, but using the difference bounds
in part (b) of Lemma~\ref{l:Jost}.
\goodbreak
b) This follows from the arguments above, using (\ref{qT_Delta}) with $\qT_1 = \qT + \varepsilon\tilde \qT$ and $\qT_2=\qT$.
\goodbreak
c) This symplectomorphism property is another application of our identity (\ref{pre_quasi_isom}). Since the right side of (\ref{pre_quasi_isom}) is real-valued, we have
\begin{align}
\Im\langle g,T_{\tilde\qT_1,\tilde \qT_2}f\rangle=\Im \langle f,W_{\tilde\qT_1,\tilde \qT_2}g\rangle,
\end{align}
for any $f,g,\tilde\qT_1,\tilde\qT_2\in L^2(\R^2)$. We use this with $\tilde \qT_1 = \tilde \qT_2 =\qT$, $f = \qT_2$ and $g = \frac{\delta\mathcal S}{\delta \qT}\Big|_{\qT}\qT_1$. Then, in view of part b), $W_{\qT,\qT}g = \qT_1$ and (\ref{symplecto}) follows.
\end{proof}

%%% ======================================================================
\section{Application to Defocusing DSII}\label{Section:DSII}
%%% ======================================================================

In this section we use the properties of the nonlinear scattering
transform $\mathcal S$ in Theorem~\ref{ThmScatteringTransform} in
order to prove the results on the defocusing DSII problem  in
Theorem~\ref{ThmDSII} as well as Theorem~\ref{ThmScattering}. We
first review the Inverse Scattering based construction of
solutions to the DSII system \eqref{DSIIdefoc}. The steps are as
follows, see \eqref{ST}:
\begin{enumerate}[label=(\roman*)]
\item We define the initial data for the scattering transform,
\[
\tT_0 = \mathcal S \qT_0
\]

\item We compute the linear evolution on the scattering transform side
\[
\tT(t,k) = e^{2i(k^2+ \bar k^2)t} \tT_0(k)
\]

\item We return to the physical space via the inverse transform $\mathcal S^{-1} = \mathcal S$,
\[
\qT(t) = \mathcal S \tT(t).
\]
\end{enumerate}

Our starting point is the classical work of Beals and Coifman
\cite{BC85}, \cite{BC} and \cite{BC89}, who show that if $q_0 \in
\mathscr S$, then $\tT_0 \in \mathscr S$ and further that $\qT(t)\in
\mathscr S$ is the unique classical solution to \eqref{DSIIdefoc}.
Our goal, on the other hand, is to show that the above algorithm
is equally valid for all $L^2$ initial data. We begin by examining
the presumptive data-to-solution map
\begin{equation}\label{d-s-map}
\qT_0 \to \qT(t,\cdot)
\end{equation}

\begin{lemma}
The data-to-solution map \eqref{d-s-map} has the following properties:
\begin{enumerate}[label=(\roman*)]
\item  Conserved mass:
\[
\|q(t,\cdot)\|_{L^2} = \|q_0\|_{L^2}.
\]
\item Continuity in time:
\[
L^2 \ni \qT_0 \to \qT(t,\cdot) \in C(\R,L^2).
\]
\item  Lipschitz property: for two $L^2$ solutions $q_1$ and $q_2$ we have
\[
\| q_1(t,\cdot) - q_2(t,\cdot)\|_{L^2} \leq
C(\|q_{01}\|_{L^2})C(\|q_{02}\|_{L^2}) \|q_{01} - q_{02}\|_{L^2}
\] for all $t$.
\item  Pointwise bound:
\[
|\qT(t,z)| \leq C(\|\qT_0\|_{L^2})M\qT^{\text{lin}}(t,z)
\]
where
\[
 \qT^{lin}(t,\cdot) = U(t)\ol{\widehat{\mathcal S\qT_0}}.
\]
\item  $L^4$ bound:
\[
\| q\|_{L^4_{t,z} } \leq C(\|q_0\|_{L^2}).
\]
\end{enumerate}
\end{lemma}

\begin{proof}
(i) This is immediate from the Plancherel identity \eqref{t1-l2}.

(ii) This is a consequence of the Lipschitz bound \eqref{t1-lip}
combined with the $L^2$ time continuity of $e^{2i(k^2+ \bar k^2)t}
\tT_0$.

(iii) This is also a consequence of the Lipschitz bound \eqref{t1-lip}.

(iv) This follows from Corollary \ref{CorBound}, noting that the
(inverse) Fourier transform of (\ref{ST}) is:
\begin{align*}
\ol{\hat \tT(t,\cdot)} = U(t)\ol{\hat\tT_0} = \qT^{lin}(t,\cdot).
\end{align*}
(v) From the Strichartz estimate for the linear flow we have
\begin{align*}
\|\qT^{lin}\|_{L^4_{t,z}}\leq C\|\hat\tT_0\|_{L^2} =
C\|\tT_0\|_{L^2} = C\|\qT_0\|_{L^2},
\end{align*}
using the Plancherel identity (\ref{t1-l2}). The bound (v) now
follows from (iv) above and the $L^4$ boundedness of the Maximal
function.
\end{proof}

This lemma shows that the $L^2$ presumptive solutions can be
viewed as the unique uniform limits of Schwartz solutions.
However, it does not yet prove that these are actual solutions to
\eqref{DSIIdefoc}.  Our next step is stated separately as it no
longer relies on the scattering transform, but rather on
perturbative dispersive analysis:

 \begin{lemma} \label{l:l4-lip}
The data-to-solution map \eqref{d-s-map} satisfies the Lipschitz bound
\[
\| \qT_1 - \qT_2\|_{L^4_{t,z}} \leq C(\|q_{01}\|_{L^2})C( \|q_{02}\|_{L^2}) \|q_{01} - q_{02}\|_{L^2}.
\]
\end{lemma}

\begin{proof}
By the property (iii) in the previous lemma and a density argument for the embedding
$\mathscr S \subset L^2$, it suffices to prove this for Schwarz data $\qT_{01}$, $\qT_{02}$.
The advantage then is that we know in addition that $q_1$ and $q_2$ are classical solutions
for \eqref{DSIIdefoc}, which we rewrite as
\[
i q_t + 2(\dzb^2+\dz^2)\qT = N(q):=  \qT L|q|^2
\]
where $L$ is a zero order multiplier, which is bounded in all $L^p$ spaces for $1 < p < \infty$.
  Then we can apply Strichartz estimates on any time interval
$I = [0,T]$ for the difference of the two solutions to obtain
\[
\begin{split}
\| \qT_1 - \qT_2\|_{L^4_{t,z} \cap L^\infty_t L^2_z[I] } \lesssim
& \
 \| \qT_{01} - \qT_{02}\|_{L^2} + \| N(\qT_1) - N(\qT_2)\|_{L^\frac43[I]} \\
\lesssim & \ \| \qT_{01} - \qT_{02}\|_{L^2} + \| \qT_1 -
\qT_2\|_{L^4_{t,z}[I]} ( \| \qT_1\|_{L^4_{t,z}[I]}^2 + \|
\qT_1\|_{L^4_{t,z}[I]}^2)
\end{split}
\]
If we have the additional property
\begin{equation}\label{l4-small}
 \| \qT_1\|_{L^4_{t,z}[I]}, \| \qT_1\|_{L^4_{t,z}[I]} \ll 1
\end{equation}
then we can absorb the second term on the right into the left hand side to obtain
\[
\| \qT_1 - \qT_2\|_{L^4_{t,z} \cap L^\infty_t L^2_z[I] } \lesssim
\| \qT_{01} - \qT_{02}\|_{L^2}
\]
To use this property we take advantage of the $L^4$  bound in part
(v) of the previous lemma in order to divide the real line into
subintervals $\R = \cup _{j \in \mathcal J} I_j$ so that the
property \eqref{l4-small} holds for all intervals $I_j$. The
number of such intervals is at most
\[
|\mathcal J| \lesssim C(\|\qT_1\|_{L^4_{t,z}})C( \|\qT_2\|_{L^4_{t,z}}) \lesssim
C(\|q_{01}\|_{L^2})C(\|q_{02}\|_{L^2})
\]
Then we apply the above argument successively on all these
intervals in order to obtain the conclusion of the Lemma.
\end{proof}

The $L^4$ Lipschitz bound can now be used in order to show that
the Inverse Scattering construction yields solutions to
\eqref{DSIIdefoc}.

\begin{lemma}
For each $\qT_0 \in L^2$ the function $q(t)$ is a solution to
\eqref{DSIIdefoc} in the sense that \eqref{Duhamel} holds.
\end{lemma}

The proof is straightforward, based on the Strichartz estimates for the linear flow.

This concludes the proof of Theorem~\ref{ThmDSII}. We now
turn our attention to Theorem~\ref{ThmScattering}. We begin with a slight
improvement of Lemma~\ref{l:l4-lip}:

\begin{lemma}
The map
\[
L^2 \ni \qT_0 \to q \in L^4_{t,z}
\]
is smooth.
\end{lemma}

\begin{proof}
This is a standard perturbative argument which we only outline.
Given a solution $\qT_1$ to DSII with initial data $\qT_{01} \in
L^2$, we seek to solve the DSII with initial data $\qT_{02}$
sufficiently close to $\qT_{01}$. Since $\qT_1 \in L^4_{t,z}$, we
can divide the real line as in the proof of Lemma~\ref{l:l4-lip}
into finitely many subintervals $I_j$ so that  $\|
\qT_1\|_{L^4_{t,z}[I_j]}$ is small.

Then  we construct the solution $\qT_2$ successively in each
subinterval by reiterating the Duhamel formula \eqref{DuhamelInt}.
This converges due to the Strichartz estimates.
\end{proof}

The next lemma establishes the existence and regularity of the wave operators $W_{\pm}$:

\begin{lemma}
The wave operators $W_{\pm}$ are well defined and locally Lipschitz in $L^2$.
\end{lemma}
\begin{proof}
We begin by using the Duhamel formula to compute
\[
U(-t) \qT(t) = \qT(0) + \int_0^t U(-s) N(\qT(s)) ds
\]
Since $q \in L^4_{t,z}$, it follows that $N(q) \in L^\frac43$.
Then by Strichartz estimates the above expression converges in
$L^2$ as $t \to \pm \infty$, and we have
\[
\qT_\pm = \lim_{t \to \pm \infty} U(-t) \qT(t) =  \qT(0) +
\int_0^{\pm \infty} U(-s) N(\qT(s)) ds.
\]
The map $\qT_0 \to \qT_\pm $ is smooth in view of the previous Lemma and Strichartz estimates.
\end{proof}

To  see that $W_{\pm} \qT_0 = \ol{\widehat{ \mathcal S q_0}}$, and
thus complete
 the proof of   Theorem~\ref{ThmScattering}, we can now argue by density.
It suffices to know that this  is true for $\qT_0 \in \mathscr S$.
This was already proved in \cite{BC}, but for the sake of
completeness we provide a self-contained argument below.

%{tT-point++}

If $\qT_0$ is Schwartz then $\tT_0$ is also Schwartz
(see\cite{BC}), and $\tT(t) = e^{2it (k^2+ \bar k^2)} \tT_0$.
 In view of the bound \eqref{tT-point++}, applied with the roles of $\qT$ and $\tT$ reversed, it
suffices to show that
\begin{equation}\label{last}
\lim_{t \to \infty} \| \Icbk  ( e_{z} e^{2it(k^2+\bar k^2)} \tT_0)
\|_{L^4} \to 0.
\end{equation}
Indeed, a direct computation shows that
\[
|  \Icbk  ( e_{z} e^{2it(k^2+\bar k^2)} \tT_0)| \lesssim
t^{-\frac12} (1+ |k|)^{-N}  (1+ t^{-\frac12} |z - 2tk|)^{-1}
\]
which has an $L^4$ norm of size $t^{-\frac14}$. This completes the
proof of \eqref{last}.

%%% ======================================================================
\section{Application to Two Inverse Boundary Value Problems}\label{Section:Calderon}
%%% ======================================================================

In this section we prove Theorems~\ref{t:bvp-solve},
\ref{reconst}, \ref{dbar-solve}, \ref{dbar-inverse}.  We begin
with the results for the boundary value problem
\eqref{dbar-model}, which we recall here:
\begin{equation*}
\left\{ \begin{array}{cc}
\dbar v -  q \bar v = 0  & \text{in } \Omega
\cr
 \Im (\nu v) =  g_0  & \text{in } \partial \Omega
\end{array}
\right.
\end{equation*}
The motivation for the study of this problem
was given in the Introduction. We start with the solvability result for this problem.

\begin{proof}[Proof of Theorem~\ref{dbar-solve}]
We consider several increasingly difficult cases:
\medskip

{\bf Case 1:  $q = 0$, $\int g_0 = 0$}. Then $v$ is holomorphic
in $\Omega$  so we can express it in the form
\[
v = \partial u
\]
with $u$ real valued in $\Omega$. Then $u$ must solve the Laplace equation
\begin{equation}
\left\{ \begin{array}{cc} \triangle u = 0   & \text{in } \Omega
\cr
 \frac{\partial u}{\partial\tau} =  -2g_0  & \text{on } \partial
 \Omega.
\end{array}
\right.
\end{equation}
Here $\tau = (-\nu_2, \nu_1)$ denotes the unit tangent vector to
$\partial\Omega$ in the counterclockwise direction. This Dirichlet
problem is uniquely solvable modulo constants if and only if $\int
g_0 = 0$, and yields a solution $u \in H^\frac32$ with
$\frac{\partial u}{\partial\nu} \in L^2$.

\bigskip

{\bf Case 2: The inhomogeneous problem.} Here we consider
the inhomogeneous problem
\begin{equation}\label{inhom}
\left\{ \begin{array}{cc} \dbar v = f_0  & \text{in } \Omega \cr
 \Im (\nu v) =  g_0  & \text{on } \partial \Omega
\end{array}
\right.
\end{equation}
with $f_0 \in L^\frac43(\Omega)$ and claim that we can solve it if
and only if
\begin{equation}\label{lin-constraint}
\int_{\Omega} \Im f_0 dz = \frac12\int_{\partial \Omega}  g_0 ds.
\end{equation}
Indeed, extending $f_0$ by zero outside $\Omega$ we can solve
$\triangle u=4f_0$ in all of $\R^2$, obtaining a solution $v_0 = \partial u  \in
W^{1,\frac43}_{loc}$, which is easily seen to have an $L^2$ trace
on the boundary. Now we are left with the homogeneous problem,
which is solvable provided that $g_0$ is in a codimension one
affine subspace. This constraint is easily seen to be
\eqref{lin-constraint} by integrating the equation \eqref{inhom}
over $\Omega$ and using the divergence theorem. We can restate the
result as follows:

\begin{lemma}
  For each $f_0 \in L^\frac43(\Omega)$ and real-valued $g_0 \in L^2(\partial
  \Omega)$ the problem
\begin{equation}\label{dbar-c}
\left\{ \begin{array}{cc}
\dbar v = f_0  & \text{in } \Omega
\cr
 \Im (\nu v) =  g_0 +c  & \text{on } \partial \Omega
\end{array}
\right.
\end{equation}
admits a unique solution $(v,c) \in H^\frac12 \times \R$.
Moreover, we have:
\begin{equation}
\|v\|_{H^{\frac12}(\Omega)} + \|v\|_{L^2(\partial\Omega)} \leq C
(\|f_0\|_{L^{\frac43}(\Omega)}+ \|g_0\|_{L^2(\partial\Omega)}).
\end{equation}
\end{lemma}
We will write
\[
v = T f_0 + B g_0
\]
where
\[
T : L^{\frac43}(\Omega) \to H^\frac12(\Omega), \qquad B:
L^2(\partial\Omega) \to H^\frac12(\Omega).
\]
Note that $c$ can be explicitly determined from $f_0$ and $g_0$.
\bigskip

{\bf  Case 3:  $q$ small.}
We first solve the counterpart of \eqref{dbar-c},
namely
\begin{equation}\label{dbar-cq}
\left\{ \begin{array}{cc}
\dbar v -  q \bar v = f_0  & \text{in } \Omega
\cr
 \Im (\nu v) =  g_0 +c  & \text{on } \partial \Omega
\end{array}
\right.
\end{equation}
This we can rewrite as
\[
v = T ( q \bar v) + Tf_0 + Bg_0
\]
which is solved by a Neumann series in $H^\frac12$.

Now we set $f_0 = 0$ and assume that $\int g_0 = 0$. The solution
we obtain above does not \emph{a priori} have $c = 0$, which is why we
need to prove that \emph{a posteriori}. Precisely, integrating by parts
against $\sigma^{-\frac12}$ and using $q = -\frac12 \partial \log \sigma$ we obtain 
\[
0 =  \Im \int_\Omega (\dbar v -  q \bar v) \sigma^{-\frac12}   dz = - \Im \int_\Omega
  v \dbar \sigma^{-\frac12} +\sigma^{-\frac12} q \bar v dz  + \half
\int_{\partial \Omega}  \Im (\nu v) ds = \half cL
\]
where $L$ is the length of $\partial \Omega$. Therefore we
conclude that $c = 0$.

\bigskip

{\bf  Case 4:  $q$ large.}
We solve again \eqref{dbar-cq}. The problem is written as
\[
v = T ( q \bar v) + Tf_0 + Bg_0
\]
The operator $v \to  T ( q \bar v) $ is compact in $H^\frac12$, as it is bounded, linear in $q$ 
and compact for smooth $q$,
so by the Fredholm alternative it remains to show that the homogeneous
problem
\[
v = T ( q \bar v)
\]
admits no nontrivial solution.

Such a solution would solve
\begin{equation}
\left\{ \begin{array}{cc}
\dbar v -  q \bar v = 0  & \text{in } \Omega
\cr
 \Im (\nu v) =  c & \text{on } \partial \Omega
\end{array}
\right.
\end{equation}
The constant $c$ must be equal to zero, as in the previous case.

From here we proceed as in the global $\dbar$ problem.
We split $q$ into $q = q_{smooth} + q_{small}$.
We seek to eliminate $q_{smooth}$ by a gauge transformation
$\phi$ which solves
\[
\dbar \phi = r:=  -\frac{\bar v}{v} \ q_{smooth}
\]
Here we need to insure that $\phi$ is real on the boundary.
So we need to solve
\begin{equation}
\left\{ \begin{array}{cc}
\dbar \phi = r   & \text{in } \Omega
\cr
 \Im \phi =  0 & \text{on } \partial \Omega
\end{array}
\right.
\end{equation}
Solving the inhomogeneous problem we are left with
\begin{equation}
\left\{ \begin{array}{cc}
\dbar \phi = 0   & \text{in } \Omega
\cr
 \Im \phi =  f & \text{on } \partial \Omega
\end{array}
\right.
\end{equation}
where we solve first for $\Im \phi$ and then $\Re \phi$ is uniquely determined modulo constants.

Now we set
\[
u = e^\phi v
\]
which solves
\begin{equation}
\left\{ \begin{array}{cc} \dbar u = q_{small} \ol u   & \text{in }
\Omega \cr
 \Im ( \nu u ) =  0 & \text{on } \partial \Omega
\end{array}
\right.
\end{equation}
We are now in the small $q$ case so $u = 0$ follows. The proof of Theorem~\ref{dbar-solve}
is concluded.
\end{proof}

As discussed in the introduction, the above proof allows us to
define a Hilbert transform operator associated to the $\dbar$
problem in $\Omega$ as
\[
L^2 \ni \Im (\nu v) \to \Hq v := \Re(\nu v) \in L^2.
\]
Next we show that the boundary data $\Hq$ uniquely determines $q$.

\begin{proof}[Proof of Theorem~\ref{dbar-inverse}]

The proof below is in the spirit of \cite{Nac88} and \cite{Nac},
also inspired by some arguments in \cite{KT} and \cite{Tam}. Let
$\qT \in L^2 ( \mathbb R^2 )$ be defined by zero extension outside
$\Omega$,
\[
\qT=\left\{ \begin{array}{cc} -\frac{1}{2}\partial\log\sigma & \text{in } \Omega \cr
0 & \text{in } \mathbb R^2 \setminus  \Omega \end{array}\right.
\]
We will  show that $\tT=\mathcal S\qT$ can be constructively determined from
knowledge of $\Hq$. The potential $\qT$ can then be recovered from
$\tT$ using the inversion Theorem \ref{ThmScatteringTransform}
(5).

For $k$ such that $M\hat{\qT}(k)<\infty$, let $m_\pm(\cdot,k)$ be
the Jost solutions of (\ref{eq_n}) constructed in Lemma
\ref{l:Jost}. We have
\begin{align*}
\tT(k) &= \frac{1}{2\pi i}\int_{\R^2} e_k(z) \ol{\qT(z)}
\Big(m_+(\cdot,k) + m_-(\cdot,k)\Big) \\
&= \frac{1}{2\pi i}\int_{\Omega} \partial
\Big(\ol{m_+}(\cdot,k) - \ol{m_-}(\cdot,k)\Big) \\
&= \frac{1}{4\pi i}\int_{\partial\Omega} \ol\nu
\Big(\ol{m_+}(\cdot,k) - \ol{m_-}(\cdot,k)\Big)
\end{align*}
Thus, it will suffice to show that one can compute the traces of
$m_\pm(\cdot,k)$ from knowledge of $\mathcal H_\qT$ on
$\partial\Omega$. Let
\begin{align}
\psi_\pm(z,k) = e^{izk}m_\pm(z,k).
\end{align}
The following lemma shows that we can obtain the trace
$\psi_+(\cdot,k)\Big|_{\partial\Omega}$ from $\Hq$.
\begin{lemma}
Let $\Omega$, $\sigma$, $\qT$ be as in Theorem \ref{reconst} and
$k$ as above. Then the function $\psi_+(z,k)$ restricted to
$z\in\C\backslash\ol\Omega$ is the unique solution of the exterior
problem
\begin{align}\label{set}
\begin{cases}
(i)&\enspace\dzb\psi_+=0\text{ in }\C\backslash\ol\Omega\\
(ii)&\enspace\psi_+(z,k)e^{-izk}-1\in L^4(\C\backslash\ol\Omega)\cap W^{1,\frac43}_{loc}\\
(iii)&\enspace\Re(\nu\psi_+\Big|_{\partial\Omega}) = \Hq
(\Im(\nu\psi_+\Big|_{\partial\Omega})).
\end{cases}
\end{align}
\end{lemma}
\begin{proof}
The main issue is to prove the uniqueness. Suppose uniqueness does
not hold. Then there exists a function $h$ with $e^{-izk}h \in
L^4(\mathbb C\setminus \Omega)$ so that $\dbar h = 0$ and
\[
\Re(\nu h) = \Hq (\Im(\nu h)) \qquad \text{in }{\partial\Omega}.
\]
Using the solvability result in Theorem~\ref{dbar-solve}
we solve the problem
\begin{equation}\label{qv}
\left\{ \begin{array}{cc} \dbar v = q v   & \text{in } \Omega \cr
 \Im ( \nu v ) =  \Im(\nu h) & \text{on } \partial \Omega.
\end{array}
\right.
\end{equation}
Then in view of the definition of $\mathcal H_q$ we must have also
\[
\Re(\nu v )  = \Re(\nu h) \qquad \text{on }{\partial\Omega}.
\]
Now let
\begin{align}
\phi = \begin{cases}  v \text{ in }\ol\Omega\\
h\text{ in }\C\backslash\ol\Omega.
\end{cases}
\end{align}
Then $m=\phi e^{-izk}\in L^4(\R^2)$. We have shown that $\phi$
(hence $m$) is continuous across $\partial\Omega$. In view of
 (\ref{set})(i) and (\ref{qv}), $m$ is a weak solution of
\begin{align}
\dzb m-e_{-k}\qT\ol m = 0
\end{align}
in all of $\R^2$. Lemma \ref{perturb} now shows that $m=0$. This
proves uniqueness. It is also clear that $\psi_+(z,k)$ restricted
to $\C\backslash\ol\Omega$ is a solution of (\ref{set}). This
completes the proof the the Lemma.
\end{proof}

For computational purposes one can use layer potentials to reduce
(\ref{set})  to a problem on $\partial\Omega$. We will not pursue
this here.

\end{proof}
Finally we return to the original Calder\'on problem with $\log
\sigma \in \dot H^1$. We begin with the solvability question.

\begin{proof}[Proof of Theorem~\ref{t:bvp-solve}]
Without loss of generality, we may assume $g$ is real-valued. 
In complex notation, the equation (\ref{bvp}) takes the form
\begin{align*}
\ol\partial(\sigma\partial u) + \partial(\sigma\ol\partial u) =0.
\end{align*}
For real valued $u$, the standard substitution $v=\sigma^\half\partial u$ then yields a solution of (\ref{dbar-model-eq}) with $\qT$ defined by (\ref{q-def}). Thus, in view of (\ref{du_dtau}), $v$ solves the boundary value problem (\ref{dbar-model}) with $g_0 = -\half \frac{\partial g}{\partial \tau}$. But this problem is uniquely solvable by Theorem \ref{dbar-solve}. Consequently $\sigma^\half\partial u$ is uniquely determined. This immediately yields $u$. We also have (see (\ref{du_dnu})):
\begin{align*}
\frac{\partial u}{\partial \nu} = 2\Re(\nu v) \in L^2(\partial\Omega),
\end{align*} 
and (\ref{DtN_H}) holds on $H^1(\partial\Omega)$.
\end{proof}

Finally, Theorem \ref{reconst} on the Calder\'on problem is now an
easy consequence of the previous results.

\begin{proof}[Proof of Theorem~\ref{reconst}]
Given $\Lambda_\sigma$, we use (\ref{DtN_H}) to determine $\mathcal H_\qT$. From Theorem \ref{dbar-inverse} we have a method to
reconstruct $\qT=-\half\partial \log \sigma$. Since $\log\sigma$
is assumed known on $\partial\Omega$, and we have determined its
gradient, we can recover this function inside $\Omega$.
\end{proof}

%%% ======================================================================
\section{Appendix}\label{Section:Appendix}
%%% ======================================================================

The following is a restatement of Lemma 4.2 in \cite{Nac_p}. It is reproduced here with the proof, for completeness:

\begin{lemma}\label{compactness}
If $a\in L^2(\C)$ then the operator $f\mapsto\Icb(a f)$ is compact on $L^r(\C)$, $2<r<\infty$.
\end{lemma}
\begin{proof}
By duality, it suffices to show that $a\Ic$ is compact on $L^p$, $1<p<2$. We have
\begin{align}\label{ineq_3}
\|a\Ic f\|_{L^p}\leq \|a\|_{L^2}\|\Ic f\|_{L^{p^*}}\leq c\|a\|_{L^2}\|f\|_{L^p}.
\end{align}
First suppose $a$ is a $C^1$ function with compact support in, say, a disk $D$. Then by the boundedness of the Beurling transform on $L^p$:
\begin{align*}
\|\nabla(a\Ic f)\|_{L^p}\leq \|\nabla a\|_{L^2}\|\Ic f\|_{L^{p^*}} + \|a\|_{L^\infty}\|\nabla\Ic f\|_{L^p}\leq c\|f\|_{L^p}.
\end{align*}
Thus, the image under $a\Ic$ of the unit ball in $L^p$ lies in 
$\{u\in L^p(D):\|u\|_{L^p}\leq c,\|\nabla u\|_{L^p}\leq c\}$, which is compact. Now let 
$a$ be arbitrary in $L^2(\C)$ and let $\{a_j\}$ be a sequence of $C^1$ functions of compact support converging to $a$ in $L^2$. The corresponding operators $a_j\Ic$ are compact and norm convergent by (\ref{ineq_3}), hence their limit, too, is compact.
\end{proof}

\bibliographystyle{plain}
\bibliography{XBib}

\end{document}